\documentclass[12pt]{article}
\usepackage[a4paper]{geometry}
\usepackage{amssymb,amsmath,amsthm}
\usepackage{graphicx,cite,color}
\usepackage{longtable,pdflscape,booktabs,caption,multicol}
\usepackage[colorlinks=true,citecolor=black,linkcolor=black,urlcolor=blue]{hyperref}

\theoremstyle{plain}
\newtheorem{theorem}{Theorem}
\newtheorem{lemma}{Lemma}
\newtheorem{corollary}{Corollary}

\theoremstyle{definition}

\newtheorem{example}{Example}

\theoremstyle{remark}
\newtheorem*{remark}{Remark}

\usepackage{float}
\usepackage{tikz}
\usetikzlibrary{arrows}
\usetikzlibrary{arrows.meta}
\usetikzlibrary{chains}
\usetikzlibrary{positioning}
\usetikzlibrary{automata,positioning,calc}
\usetikzlibrary{decorations}
\usetikzlibrary{decorations.shapes}
\usetikzlibrary{decorations.markings}
\tikzset{
    edge/.style={-{Latex[scale=1.7]}},
    dedge/.style={{Latex[scale=1.7]}-{Latex[scale=1.7]}},
}

\title{Assigned rational functions of a rooted tree}

\author{Ivan Damnjanovi\'c\thanks{The author is supported by Diffine LLC.}\\
\small University of Ni\v s, Faculty of Electronic Engineering\\[-0.4ex]
\small\tt ivan.damnjanovic@elfak.ni.ac.rs\\
\small Diffine LLC\\[-0.4ex]
\small\tt ivan@diffine.com}

\begin{document}

\maketitle

\begin{abstract}
    We investigate the spectral properties of rooted trees with the intention of improving the currently existing results that deal with this matter. The concept of an assigned rational function is recursively defined for each vertex of a rooted tree. Afterwards, two mathematical formulas are given which show how the characteristic polynomials of the adjacency and Laplacian matrix can be represented as products of the aforementioned rational functions. In order to demonstrate their general use case scenario, the obtained formulas are subsequently implemented on balanced trees, with a special focus on the Bethe trees. In the end, some of the previously derived results are used in order to construct a tree merging procedure which preserves the spectra of all of the starting trees.

\end{abstract}

\section{Introduction}\label{sc_introduction}

Let $G$ be a simple graph. We will use $V(G)$ and $n(G)$ to represent the vertex set and order of this graph, respectively. Moreover, we will denote the degree of each vertex $v \in V(G)$ by $d(v)$. Also, we will signify the adjacency and Laplacian matrix of the graph $G$ by $A(G)$ and $L(G)$, respectively. Here, we assume that the rows and columns of $A(G)$ and $L(G)$ correspond to the vertices $u_1, u_2, \ldots, u_{n(G)} \in V(G)$, in this order. We shall denote the characteristic polynomials of these two matrices by $P(G, x) = \det(xI - A(G))$ and $Q(G, x) = \det(xI - L(G))$. Also, we will use $\sigma(G)$ to signify the spectrum of $G$, regarded as a multiset composed of the eigenvalues of $A(G)$, as well as $\sigma^*(G)$ to signify the set of all the distinct eigenvalues of $A(G)$. Finally, we shall use $\mathcal{E}(G)$ in order to denote the energy of the graph $G$. Here, the graph energy represents the sum of absolute values of all the eigenvalues of $A(G)$, as introduced by Gutman in \cite{gutman_energy}.

We know that a tree represents a simple graph which is both acyclic and connected. Let $T$ be a rooted tree whose root is denoted by $r(T)$. We will enumerate the levels of $T$ by $1, 2, 3, \ldots, l(T)$ so that $r(T)$ is located on level $1$ and $l(T) - 1$ represents the eccentricity of $r(T)$. Also, we will use $V(T, j)$ and $n(T, j)$ to signify the set of all of the vertices located on level $j$ and their total number, respectively, for all the $1 \le j \le l(T)$. For convenience, we will also define $n(T, 0) = 0$. \linebreak Finally, we shall use $c(v)$ to denote the set of all of the children of some given vertex $v \in V(T)$. 

A balanced tree is a rooted tree such that all of the vertices on the same level have an equal degree. It is clear that all of its leaves must be on the last level $l(T)$. We will consider the Bethe tree $\mathcal{B}_{d, k}$ to be the balanced tree such that
\begin{itemize}
    \item $l(\mathcal{B}_{d, k}) = k$;
    \item each vertex has $d-1$ children, besides the vertices on the last level, which obviously have none;
\end{itemize}
as defined by Heilmann and Lieb in \cite{heilmann_lieb}. In this paper, we also define the anti-factorial tree $\mathcal{A}_k$ to be the balanced tree such that
\begin{itemize}
    \item $l(\mathcal{A}_k) = k$;
    \item each vertex on level $j$ has exactly $k-j$ children.
\end{itemize}

At a general level, this paper deals with the spectral properties of rooted trees. The central results we have obtained are stated in the following two corollaries.
\begin{corollary}\label{adj_cor}
    Let $T$ be an arbitrary rooted tree. If we recursively assign a rational function $\mathcal{G}(v, x) \in \mathbb{Z}(x)$ to each of its vertices $v \in V(T)$ by using the expression
    \begin{align}\label{adj_cor_rec}
        \mathcal{G}(v, x) &= x - \sum_{w \in c(v)} \dfrac{1}{\mathcal{G}(w, x)} \qquad (\forall v \in V(T)) ,
    \end{align}
    then
    \begin{align}\label{adj_cor_res}
        P(T, x) &= \prod_{v \in V(T)} \mathcal{G}(v, x) .
    \end{align}
\end{corollary}
\begin{corollary}\label{lap_cor}
    Let $T$ be an arbitrary rooted tree. If we recursively assign a rational function $\mathcal{H}(v, x) \in \mathbb{Z}(x)$ to each of its vertices $v \in V(T)$ by using the expression
    \begin{align}\label{lap_cor_rec}
        \mathcal{H}(v, x) &= x - d(v) - \sum_{w \in c(v)} \dfrac{1}{\mathcal{H}(w, x)} \qquad (\forall v \in V(T)) ,
    \end{align}
    then
    \begin{align}\label{lap_cor_res}
        Q(T, x) &= \prod_{v \in V(T)} \mathcal{H}(v, x) .
    \end{align}
\end{corollary}

As we can see, Corollaries \ref{adj_cor} and \ref{lap_cor} provide a quick way of computing the characteristic polynomials of the adjacency and Laplacian matrix of rooted trees, as well as the corresponding spectra. The method is especially convenient when the given tree has a high degree of regularity in its structure, due to the fact that the computations which arise from Corollaries \ref{adj_cor} and \ref{lap_cor} do not rely on any kind of matrix manipulaton, but on dealing with certain polynomial sequences. This makes the above method an improvement over some of the previously used methods, such as the ones in \cite{rojo_1, rojo_2, heydari, bokhary_energy}.

The remainder of this paper is structured as follows. In Section \ref{sc_main}, we give the general definition of an assigned rational function corresponding to a vertex of a rooted tree. In Subsection \ref{subsc_main_results}, we state and provide a detailed proof of a theorem which explains how the characteristic polynomials of a wide array of matrices can be computed by implementing the aforementioned rational functions. Corollaries \ref{adj_cor} and \ref{lap_cor} will directly follow from this theorem. Subsection \ref{subsc_small} will serve to demonstrate how these two corollaries can be implemented in order to manually compute the $P(T, x)$ and $Q(T, x)$ polynomials of a given rooted tree. In order to avoid making the computations numerically tedious, there will be two provided examples which both contain a rooted tree of small order.

Subsequently, Section \ref{sc_balanced} will show how the two given corollaries can be used on potentially bigger rooted trees which have a regular structure. Subsection \ref{subsc_balanced_main} will revolve around computing the characteristic polynomials $P(T, x)$ and $Q(T, x)$ of balanced trees, as well as the corresponding set of distinct eigenvalues $\sigma^{*}(T)$. Afterwards, Subsection \ref{subsc_bethe} will rely on the previously derived results in order to investigate the spectral properties of Bethe trees. Here, we will determine the characteristic polynomial $P(\mathcal{B}_{d, k}, x)$, together with the set of distinct eigenvalues $\sigma^{*}(\mathcal{B}_{d, k})$ and the graph energy $\mathcal{E}(\mathcal{B}_{d, k})$, corresponding to each Bethe tree. Similarly, Subsection \ref{subsc_antifactorial} will deal with the spectral properties of anti-factorial trees. In this subsection we shall compute the characteristic polynomial $P(\mathcal{A}_k, x)$ and determine the set of distinct eigenvalues $\sigma^{*}(\mathcal{A}_k)$ for each such tree.

\newpage
Finally, in Section \ref{sc_merging} our goal will be to demonstrate a tree merging procedure which preserves the spectra of all of the starting rooted trees, thereby showing another potential use case of Corollary \ref{adj_cor}. This section will focus on giving an exact formulation of the aforementioned tree merging method, together with the formal mathematical proof of its validity.

\section{Assigned rational functions}\label{sc_main}
\subsection{Main results}\label{subsc_main_results}

Let $\beta(u_1), \beta(u_2), \ldots, \beta(u_{n(T)})$ be an arbitrarily chosen sequence of integers which correspond to the vertices $u_1, u_2, \ldots, u_{n(T)}$ of a rooted tree $T$, respectively. For such a given sequence, we will define two additional matrices
\begin{align}
    \label{b_1_formula}B_1(T) &= A(T) + \mathrm{diag}(\beta(u_1), \beta(u_2), \ldots, \beta(u_{n(T)})) ,\\
    \label{b_2_formula}B_2(T) &= -A(T) + \mathrm{diag}(\beta(u_1), \beta(u_2), \ldots, \beta(u_{n(T)})) .
\end{align}
In this paper, we introduce a recursive formula which helps compute the characteristic polynomials of the $B_1(T)$ and $B_2(T)$ matrices for any given $\beta$-sequence. This formula is based on assigning a rational function to each vertex of a rooted tree in the bottom-up manner, by using the rational functions previously assigned to its children. The corresponding theorem is given below.
\begin{theorem}\label{main_theorem}
    Let $T$ be any rooted tree with an arbitrarily chosen $\beta$-sequence. If we recursively assign a rational function $\mathcal{F}(v, x) \in \mathbb{Z}(x)$ to each of its vertices $v \in V(T)$ by using the expression
    \begin{align}\label{recursive}
        \mathcal{F}(v, x) &= x - \beta(v) - \sum_{w \in c(v)} \dfrac{1}{\mathcal{F}(w, x)} \qquad (\forall v \in V(T)) ,
    \end{align}
    then
    \begin{align}\label{main_formula}
        \det(xI - B_1(T)) = \det(xI - B_2(T)) &= \prod_{v \in V(T)} \mathcal{F}(v, x) .
    \end{align}
\end{theorem}

First of all, it is clear how Theorem \ref{main_theorem} can be implemented in order to yield both Corollary \ref{adj_cor} and \ref{lap_cor}.

\bigskip
\noindent
{\em Proof of Corollary \ref{adj_cor}}.\quad If we set $\beta(v) = 0$ for all the $v \in V(T)$, we then get $B_1(T) = A(T)$, according to Eq.\ (\ref{b_1_formula}). This means that $P(T, x) = \det(xI - B_1(T))$. By comparing Eq.\ (\ref{adj_cor_rec}) to Eq.\ (\ref{recursive}), we conclude that Eq.\ (\ref{adj_cor_res}) immediately follows from Eq.\ (\ref{main_formula}). \qed

\bigskip
\noindent
{\em Proof of Corollary \ref{lap_cor}}.\quad By setting $\beta(v) = d(v)$ for all the $v \in V(T)$, we obtain $B_2(T) = L(T)$, according to Eq.\ (\ref{b_2_formula}). This implies $Q(T, x) = \det(xI - B_2(T))$. By comparing Eq.\ (\ref{lap_cor_rec}) to Eq.\ (\ref{recursive}), we see that Eq.\ (\ref{lap_cor_res}) follows directly from Eq.\ (\ref{main_formula}). \qed

\bigskip
In the remainder of this subsection, we will give a complete proof of Theorem~\ref{main_theorem}. In order to make the logical reasoning more concise and easier to follow, we will start off with some preliminary remarks, then state and prove two auxiliary lemmas which will help us finish the entire proof afterwards.

First of all, it is clear that the matrix $B(T)$ is real and symmetric. Given the fact that any permutation matrix $P \in \mathbb{R}^{n(T) \times n(T)}$ is orthogonal, the matrix $P^T B(T) P$ must also be real and symmetric, as well as similar to $B(T)$. This means that regardless of how we order the vertices of $T$, the corresponding matrix $B(T)$ will have the same characteristic polynomial. Without loss of generality, we will assume that the vertices $u_1, u_2, \ldots, u_{n(T)}$ which correspond to the rows and columns of $B(T)$, in this order, are such that the first $n(T, l(T))$ of them are all from level $l(T)$, then the following $n(T, l(T)-1)$ are all from level $l(T)-1$, and so on. Bearing this in mind, the matrix $A(T)$ obtains the tridiagonal block form
\begin{align*}
    A(T) &= \begin{bmatrix}
        O & D_{l(T)} & O & \cdots & O & O\\
        D_{l(T)}^T & O & D_{l(T)-1} & \cdots & O & O\\
        O & D_{l(T)-1}^T & O & \cdots & O & O\\
        \vdots & \vdots & \vdots & \ddots & \vdots & \vdots\\
        O & O & O & \cdots & O & D_2\\
        O & O & O & \cdots & D_2^T & O
    \end{bmatrix},
\end{align*}
where the $D_j \in \mathbb{Z}^{n(T, j) \times n(T, j-1)},\ j = \overline{2, l(T)}$ are all binary and have exactly one $1$ per row. This value of $1$ signifies which vertex from level $j-1$ is the unique parent of each vertex from level $j$.

Consequently, we get
\allowdisplaybreaks
\begin{align*}
    B_1(T) &= \begin{bmatrix}
        C_{l(T)} & D_{l(T)} & O & \cdots & O & O\\
        D_{l(T)}^T & C_{l(T)-1} & D_{l(T)-1} & \cdots & O & O\\
        O & D_{l(T)-1}^T & C_{l(T)-2} & \cdots & O & O\\
        \vdots & \vdots & \vdots & \ddots & \vdots & \vdots\\
        O & O & O & \cdots & C_2 & D_2\\
        O & O & O & \cdots & D_2^T & C_1
    \end{bmatrix},\\
    B_2(T) &= \begin{bmatrix}
        C_{l(T)} & -D_{l(T)} & O & \cdots & O & O\\
        -D_{l(T)}^T & C_{l(T)-1} & -D_{l(T)-1} & \cdots & O & O\\
        O & -D_{l(T)-1}^T & C_{l(T)-2} & \cdots & O & O\\
        \vdots & \vdots & \vdots & \ddots & \vdots & \vdots\\
        O & O & O & \cdots & C_2 & -D_2\\
        O & O & O & \cdots & -D_2^T & C_1
    \end{bmatrix},
\end{align*}
\interdisplaylinepenalty=10000
where the $C_j \in \mathbb{Z}^{n(T, j) \times n(T, j)},\ j = \overline{1, l(T)}$ matrices are all diagonal and such that the element $\beta(v)$ corresponds to the vertex $v \in V(T)$ in the appropriate matrix. This immediately leads us to
\begin{align}
    \nonumber\det(&xI - B_1(T)) =\\
    \label{det_b1}&= \begin{vmatrix}
        xI - C_{l(T)} & -D_{l(T)} & O & \cdots & O & O\\
        -D_{l(T)}^T & xI - C_{l(T)-1} & -D_{l(T)-1} & \cdots & O & O\\
        O & -D_{l(T)-1}^T & xI - C_{l(T)-2} & \cdots & O & O\\
        \vdots & \vdots & \vdots & \ddots & \vdots & \vdots\\
        O & O & O & \cdots & xI - C_2 & -D_2\\
        O & O & O & \cdots & -D_2^T & xI - C_1
    \end{vmatrix},
\end{align}
as well as
\begin{align}
    \nonumber\det(&xI - B_2(T)) =\\
    \label{det_b2}&= \begin{vmatrix}
        xI - C_{l(T)} & D_{l(T)} & O & \cdots & O & O\\
        D_{l(T)}^T & xI - C_{l(T)-1} & D_{l(T)-1} & \cdots & O & O\\
        O & D_{l(T)-1}^T & xI - C_{l(T)-2} & \cdots & O & O\\
        \vdots & \vdots & \vdots & \ddots & \vdots & \vdots\\
        O & O & O & \cdots & xI - C_2 & D_2\\
        O & O & O & \cdots & D_2^T & xI - C_1
    \end{vmatrix}.
\end{align}
Here, the matrices $xI - B_1(T)$ and $xI - B_2(T)$ are such that all of their elements are integer polynomials in $x$, i.e.\ members of the integral domain $\mathbb{Z}[x]$. From abstract algebra we know that the corresponding field of fractions of $\mathbb{Z}[x]$ is actually the field of rational functions with integer coefficients, i.e.\ $\mathbb{Z}(x)$. Henceforth we shall interpret all of the elements of $xI - B_1(T)$ and $xI - B_2(T)$ as members of $\mathbb{Z}(x)$, and thus view $xI - B_1(T)$ and $xI - B_2(T)$ as matrices over the field $\mathbb{Z}(x)$.

Although none of the positive degree polynomials are invertible in $\mathbb{Z}[x]$, they are all invertible in $\mathbb{Z}(x)$. This is the primary reason why the given approach is useful. It will enable us to perform certain block row matrix transformations on $xI - B_1(T)$ and $xI - B_2(T)$ in order to compute the necessary characteristic polynomials, as we shall soon see. We now state and prove two auxiliary lemmas which are necessary to complete the proof of Theorem \ref{main_theorem}.

\begin{lemma}\label{lemma_start}
    For each vertex $v \in V(T)$, the assigned rational function $\mathcal{F}(v, x)$ can be represented as a fraction of polynomials
    \begin{align*}
        \mathcal{F}(v, x) &= \dfrac{\mathcal{F}_1(v, x)}{\mathcal{F}_2(v, x)} ,
    \end{align*}
    so that $\mathcal{F}_1(v, x), \mathcal{F}_2(v, x) \in \mathbb{Z}[x]$ are both monic polynomials which satisfy
    \begin{align*}
        \deg \mathcal{F}_1(v, x) &= \deg \mathcal{F}_2(v, x) + 1 .
    \end{align*}
\end{lemma}
\begin{proof}
    We will prove the lemma via mathematical induction. If we pick an arbitrary vertex $v \in n(T, l(T))$, we know that it must be a leaf, hence its assigned rational function is the linear polynomial $x - \beta(v)$, according to Eq.\ (\ref{recursive}). This rational function obviously satisfies the lemma statement, given the fact that $\mathcal{F}(v, x) = \dfrac{x - \beta(v)}{1}$. Now suppose that the statement holds for all of the vertices on level $j+1$, where $1 \le j \le l(T) - 1$. We will complete the proof by showing that it must hold for each vertex on level $j$ as well.
    
    Let $v \in V(T, j)$ be an arbitrary vertex on level $j$. From the induction hypothesis, we know that all of its children satisfy the lemma statement, i.e.\ for each $w \in c(v)$ we have
    \begin{align*}
        \mathcal{F}(w, x) &= \dfrac{\mathcal{F}_1(w, x)}{\mathcal{F}_2(w, x)} ,
    \end{align*}
    where $\mathcal{F}_1(w, x), \mathcal{F}_2(w, x) \in \mathbb{Z}[x]$ are monic polynomials such that
    \begin{align*}
        \deg \mathcal{F}_1(w, x) = \deg \mathcal{F}_2(w, x) + 1 .
    \end{align*}
    We further have
    \begin{align*}
        \mathcal{F}(v, x) &= x - \beta(v) - \sum_{w \in c(v)} \dfrac{1}{\mathcal{F}(w, x)}\\
        &= x - \beta(v) - \sum_{w \in c(v)} \dfrac{\mathcal{F}_2(w, x)}{\mathcal{F}_1(w, x)}\\
        &= \dfrac{(x - \beta(v)) \displaystyle\prod_{w \in c(v)}\mathcal{F}_1(w, x) - \displaystyle\sum_{w \in c(v)} \left( \mathcal{F}_2(w, x) \displaystyle\prod_{t \in c(v), t \neq w} \mathcal{F}_1(t, x) \right)} {\displaystyle\prod_{w \in c(v)} \mathcal{F}_1(w, x)} .
    \end{align*}
    If we write
    \begin{align*}
        \mathcal{F}_1(v, x) &= (x - \beta(v)) \displaystyle\prod_{w \in c(v)}\mathcal{F}_1(w, x) - \displaystyle\sum_{w \in c(v)} \left( \mathcal{F}_2(w, x) \displaystyle\prod_{t \in c(v), t \neq w} \mathcal{F}_1(t, x) \right) ,\\
        \mathcal{F}_2(v, x) &= \displaystyle\prod_{w \in c(v)} \mathcal{F}_1(w, x) ,
    \end{align*}
    we then obtain $\mathcal{F}(v, x) = \dfrac{\mathcal{F}_1(v, x)}{\mathcal{F}_2(v, x)}$, where $\mathcal{F}_1(v, x), \mathcal{F}_2(v, x) \in \mathbb{Z}[x]$ are clearly both monic and satisfy $\deg \mathcal{F}_1(v, x) = \deg \mathcal{F}_2(v, x) + 1$ as well. This implies that the lemma statement holds for an arbitrary $v \in V(T, j)$, which completes the proof via mathematical induction.
\end{proof}

Lemma \ref{lemma_start} directly shows that $\mathcal{F}(v, x) \in \mathbb{Z}(x)$ is invertible for each $v \in V(T)$. This is important because it makes the definition itself of the assigned rational functions valid, given the fact that the rational function assigned to each vertex demands that the rational functions assigned to all of its children have an inverse, according to Eq.\ (\ref{recursive}). We will make further use of this property in the next lemma.

\begin{lemma}\label{second_lemma}
    For each $1 \le j \le l(T)$, let $R_j$ be the diagonal square matrix of order $n(T, j)$ over the field $\mathbb{Z}(x)$, such that the diagonal entry corresponding to the vertex $v \in V(T, j)$ is equal to $\mathcal{F}(v, x)$, according to the predetermined order of vertices on level $j$. Then, the following equation holds
    \begin{align}\label{matrix_rec}
        R_j &= xI - C_j - D_{j+1}^T R_{j+1}^{-1} D_{j+1} ,
    \end{align}
    for all the $j = \overline{1, l(T)-1}$.
\end{lemma}
\begin{proof}
    First of all, the matrix $R_{j+1}$ is invertible, due to the fact that it is diagonal, and all of its diagonal entries are invertible, as a direct consequence of Lemma \ref{lemma_start}. This means that the formula Eq.\ (\ref{matrix_rec}) is valid to begin with. Furthermore, $R_{j+1}^{-1}$ must be equal to the diagonal matrix such that the element corresponding to the vertex $v \in V(T, j+1)$ is $\dfrac{1}{\mathcal{F}(v, x)}$. If we denote $Z_j = D_{j+1}^T R_{j+1}^{-1} D_{j+1}$, from basic matrix multiplication we get
    \begin{align}\nonumber
        [Z_j]_{\alpha, \beta} &= \sum_{i = 1}^{n(T, j+1)}\sum_{h = 1}^{n(T, j+1)}[D_{j+1}^T]_{\alpha, i} \ [R_{j+1}^{-1}]_{i, h} \ [D_{j+1}]_{h, \beta}\\
        \label{mult_sum_1}&= \sum_{i, h = 1}^{n(T, j+1)}[D_{j+1}^T]_{\alpha, i} \ [R_{j+1}^{-1}]_{i, h} \ [D_{j+1}]_{h, \beta} \, .
    \end{align}
    Given the fact that the matrix $R_{j+1}^{-1}$ is diagonal, we conclude that all of the sum terms in Eq.\ (\ref{mult_sum_1}) where $i \neq h$ amount to zero. This leads us to
    \begin{align}\nonumber
        [Z_j]_{\alpha, \beta} &= \sum_{i = 1}^{n(T, j+1)}[D_{j+1}^T]_{\alpha, i} \ [R_{j+1}^{-1}]_{i, i} \ [D_{j+1}]_{i, \beta}\\
        \label{mult_sum_2} &= \sum_{i = 1}^{n(T, j+1)}[D_{j+1}]_{i, \alpha} \ [D_{j+1}]_{i, \beta} \ [R_{j+1}^{-1}]_{i, i} \, .
    \end{align}
    Since the matrix $D_{j+1}$ has exactly one non-zero element per row, it immediately follows from Eq.\ (\ref{mult_sum_2}) that $[Z_j]_{\alpha, \beta} = 0$ whenever $\alpha \neq \beta$. This means that the matrix $Z_j$ is necessarily diagonal. We know that $D_{j+1}$ is also binary, which means that
    \begin{align}\nonumber
        [Z_j]_{\alpha, \alpha} &= \sum_{i = 1}^{n(T, j+1)}[D_{j+1}]_{i, \alpha} \ [D_{j+1}]_{i, \alpha} \ [R_{j+1}^{-1}]_{i, i}\\
        \label{mult_sum_3} &= \sum_{i = 1}^{n(T, j+1)}[D_{j+1}]_{i, \alpha} \ [R_{j+1}^{-1}]_{i, i} \, .
    \end{align}
    From Eq.\ (\ref{mult_sum_3}) it becomes easy to see that $Z_j$ must be the diagonal matrix of order $n(T, j)$ such that the diagonal entry corresponding to the vertex $v \in V(T, j)$ is equal to $\displaystyle\sum_{w \in c(v)} \dfrac{1}{\mathcal{F}(w, x)}$. In that case, the matrix $xI - C_j - D_{j+1}^T R_{j+1}^{-1} D_{j+1}$ must be diagonal and such that the diagonal entry corresponding to the vertex $v \in V(T, j)$ is equal to $x - \beta(v) - \displaystyle\sum_{w \in c(v)} \dfrac{1}{\mathcal{F}(w, x)}$. Hence, this matrix must be equal to $R_j$.
\end{proof}

We now have all of the tools necessary to complete the proof of Theorem~\ref{main_theorem}.

\bigskip
\noindent
{\em Proof of Theorem \ref{main_theorem}}.\quad We will prove only the
$$\det(xI - B_1(T)) = \prod_{v \in V(T)} \mathcal{F}(v, x)$$
part of Eq.\ (\ref{main_formula}). The second half is proved absolutely analogously, so the corresponding proof will be omitted.

\newpage
Given the fact that all of the vertices on level $l(T)$ are leaves, it is clear that $R_{l(T)} = xI - C_{l(T)}$. From Eq.\ (\ref{det_b1}) we get
    \begin{align*}
        \det(&xI - B_1(T)) =\\
        &= \begin{vmatrix}
            R_{l(T)} & -D_{l(T)} & O & \cdots & O & O\\
            -D_{l(T)}^T & xI - C_{l(T)-1} & -D_{l(T)-1} & \cdots & O & O\\
            O & -D_{l(T)-1}^T & xI - C_{l(T)-2} & \cdots & O & O\\
            \vdots & \vdots & \vdots & \ddots & \vdots & \vdots\\
            O & O & O & \cdots & xI - C_2 & -D_2\\
            O & O & O & \cdots & -D_2^T & xI - C_1
        \end{vmatrix}.
    \end{align*}
    Due to Lemma \ref{second_lemma}, we know that the matrix $R_{l(T)}$ is invertible. We can now multiply the first block row with $D_{l(T)}^T R_{l(T)}^{-1}$ to the left and add it to the second block row, in order to obtain
    \begin{align*}
        \det(xI - B_1(T)) &= \begin{vmatrix}
            R_{l(T)} & -D_{l(T)} & O & \cdots & O & O\\
            O & R_{l(T)-1} & -D_{l(T)-1} & \cdots & O & O\\
            O & -D_{l(T)-1}^T & xI - C_{l(T)-2} & \cdots & O & O\\
            \vdots & \vdots & \vdots & \ddots & \vdots & \vdots\\
            O & O & O & \cdots & xI - C_2 & -D_2\\
            O & O & O & \cdots & -D_2^T & xI - C_1
        \end{vmatrix}.
    \end{align*}
    Here, we have used the fact that $R_{l(T)-1} = xI - C_{l(T)-1} - D_{l(T)}^T R_{l(T)}^{-1}D_{l(T)}$, which follows from Lemma \ref{second_lemma}. We can now multiply the second block row with $D_{l(T)-1}^{T}R_{l(T)-1}^{-1}$ to the left and add it to the third block row, thus getting
    \begin{align*}
        \det(xI - B_1(T)) &= \begin{vmatrix}
            R_{l(T)} & -D_{l(T)} & O & \cdots & O & O\\
            O & R_{l(T)-1} & -D_{l(T)-1} & \cdots & O & O\\
            O & O & R_{l(T)-2} & \cdots & O & O\\
            \vdots & \vdots & \vdots & \ddots & \vdots & \vdots\\
            O & O & O & \cdots & xI - C_2 & -D_2\\
            O & O & O & \cdots & -D_2^T & xI - C_1
        \end{vmatrix},
    \end{align*}
    thanks to $R_{l(T)-2} = xI - C_{l(T)-2} - D_{l(T)-1}^T R_{l(T)-1}^{-1}D_{l(T)-1}$, which is also a consequence of Lemma \ref{second_lemma}. By repeating the same process until the last row, we conclude that
    \begin{align*}
        \det(xI - B_1(T)) &= \begin{vmatrix}
            R_{l(T)} & -D_{l(T)} & O & \cdots & O & O\\
            O & R_{l(T)-1} & -D_{l(T)-1} & \cdots & O & O\\
            O & O & R_{l(T)-2} & \cdots & O & O\\
            \vdots & \vdots & \vdots & \ddots & \vdots & \vdots\\
            O & O & O & \cdots & R_2 & -D_2\\
            O & O & O & \cdots & O & R_1
        \end{vmatrix}.
    \end{align*}
    It immediately follows that
    \begin{align*}
        \det(xI - B_1(T)) &= \prod_{j = 1}^{l(T)} \det(R_j) .
    \end{align*}
    However, from the definition of $R_j$ we know that $\det(R_j) = \displaystyle\prod_{v \in V(T, j)} \mathcal{F}(v, x)$, which directly gives us
    \begin{align*}
        \det(xI - B_1(T)) &= \prod_{j = 1}^{l(T)} \prod_{v \in V(T, j)} \mathcal{F}(v, x)\\
        &= \prod_{v \in V(T)} \mathcal{F}(v, x) .
    \end{align*}
    \qed

\begin{remark}
    It is also possible to allow the $\beta$-sequence to be a sequence of real numbers, instead of strictly integers. In this case, Theorem \ref{main_theorem} would continue to hold, with the sole difference being that the assigned rational functions $\mathcal{F}(v, x)$ would be members of $\mathbb{R}(x)$ instead of $\mathbb{Z}(x)$.
\end{remark}

\subsection{Usage on small trees}\label{subsc_small}

In this subsection, we demonstrate the usage of Corollaries \ref{adj_cor} and \ref{lap_cor} while computing the characteristic polynomials of the adjacency and Laplacian matrix of two concrete rooted trees of small order.

\begin{example}\label{example_1}
    Let $T$ be the rooted tree given on Figure \ref{small_tree_1}.
\begin{figure}
    \centering
    \begin{tikzpicture}
        \node[state, minimum size=0.25cm, thick, label=below:$x$] (1) at (0, 0) {$u_1$};
        \node[state, minimum size=0.25cm, thick, label=below:$x$] (2) at (1.2, 0) {$u_2$};
        \node[state, minimum size=0.25cm, thick, label=below:$x$] (3) at (2.4, 0) {$u_3$};
        \node[state, minimum size=0.25cm, thick, label=below:$x$] (4) at (3.6, 0) {$u_4$};
        \node[state, minimum size=0.25cm, thick, label=below:$x$] (5) at (4.8, 0) {$u_5$};
        
        \node[state, minimum size=0.25cm, thick, label=left:$x-\frac{2}{x}$] (6) at (0.6, 1.5) {$u_6$};
        \node[state, minimum size=0.25cm, thick, label=right:$x-\frac{3}{x}$] (7) at (3.6, 1.5) {$u_7$};

        \node[state, minimum size=0.25cm, thick, label=above:$x-\frac{1}{x-\frac{2}{x}}-\frac{1}{x-\frac{3}{x}}$] (8) at (2.1, 3) {$u_8$};

        \path[thick] (1) edge (6);
        \path[thick] (2) edge (6);
        \path[thick] (3) edge (7);
        \path[thick] (4) edge (7);
        \path[thick] (5) edge (7);
        \path[thick] (6) edge (8);
        \path[thick] (7) edge (8);
    \end{tikzpicture}
    \caption{The rooted tree $T$ from Example \ref{example_1}, along with its assigned rational functions $\mathcal{G}$ used to compute $P(T, x)$.}
    \label{small_tree_1}
\end{figure}
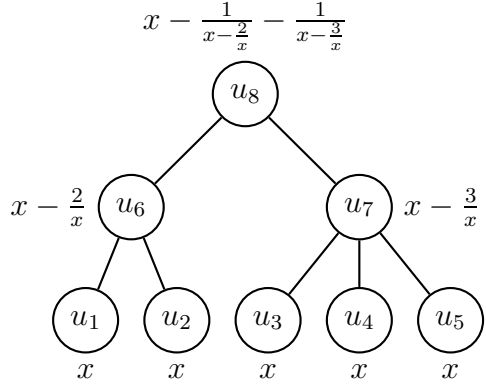
According to Corollary~\ref{adj_cor} and Figure~\ref{small_tree_1}, we obtain
\begin{align*}
    \mathcal{G}(u_1, x) &= \mathcal{G}(u_2, x) = \mathcal{G}(u_3, x) = \mathcal{G}(u_4, x) = \mathcal{G}(u_5, x) = x ,\\
    \mathcal{G}(u_6, x) &= x - \frac{2}{x} = \dfrac{x^2-2}{x} ,\\
    \mathcal{G}(u_7, x) &= x - \frac{3}{x} = \dfrac{x^2-3}{x} ,\\
    \mathcal{G}(u_8, x) &= x - \frac{x}{x^2-2} - \frac{x}{x^2-3} = \frac{x^5-7x^3 +11x}{(x^2-2)(x^2-3)} .
\end{align*}
This leads us to
\begin{align*}
    P(T, x) &= \prod_{j=1}^{8} \mathcal{G}(u_j, x)\\
    &= x^5 \ \frac{x^2-2}{x} \ \frac{x^2-3}{x} \ \frac{x^5-7x^3 +11x}{(x^2-2)(x^2-3)}\\
    &= x^4 (x^4 - 7x^2 + 11) .
\end{align*}
Thus, we have managed to compute the characteristic polynomial of the adjacency matrix of $T$. From here, it is straightforward to see that $A(T)$ has four simple eigenvalues $\pm \sqrt{\dfrac{7 \pm \sqrt{5}}{2}}$, as well as the eigenvalue $0$ whose multiplicity is four. Also, the energy of $T$ must be equal to
\begin{align*}
    \mathcal{E}(T) &= 2 \sqrt{\dfrac{7 + \sqrt{5}}{2}} + 2 \sqrt{\dfrac{7 - \sqrt{5}}{2}}\\
    &= \sqrt{14 + 2 \sqrt{5}} + \sqrt{14 - 2 \sqrt{5}} \, .
\end{align*}

\newpage
If our goal is to compute $Q(T, x)$ and the spectrum of $L(T)$ instead, we can implement Corollary \ref{lap_cor} in a very similar manner, as depicted on Figure \ref{small_tree_2}.
\begin{figure}
    \centering
    \begin{tikzpicture}
        \node[state, minimum size=0.25cm, thick, label=below:$x-1$] (1) at (0, 0) {$u_1$};
        \node[state, minimum size=0.25cm, thick, label=below:$x-1$] (2) at (1.4, 0) {$u_2$};
        \node[state, minimum size=0.25cm, thick, label=below:$x-1$] (3) at (2.8, 0) {$u_3$};
        \node[state, minimum size=0.25cm, thick, label=below:$x-1$] (4) at (4.2, 0) {$u_4$};
        \node[state, minimum size=0.25cm, thick, label=below:$x-1$] (5) at (5.6, 0) {$u_5$};
        
        \node[state, minimum size=0.25cm, thick, label=left:$x-3-\frac{2}{x-1}$] (6) at (0.7, 1.5) {$u_6$};
        \node[state, minimum size=0.25cm, thick, label=right:$x-4-\frac{3}{x-1}$] (7) at (4.2, 1.5) {$u_7$};

        \node[state, minimum size=0.25cm, thick, label=above:$x-2-\frac{1}{x-3-\frac{2}{x-1}}-\frac{1}{x-4-\frac{3}{x-1}}$] (8) at (2.45, 3) {$u_8$};

        \path[thick] (1) edge (6);
        \path[thick] (2) edge (6);
        \path[thick] (3) edge (7);
        \path[thick] (4) edge (7);
        \path[thick] (5) edge (7);
        \path[thick] (6) edge (8);
        \path[thick] (7) edge (8);
    \end{tikzpicture}
    \caption{The rooted tree $T$ from Example \ref{example_1}, together with the assigned rational functions $\mathcal{H}$ which are used to compute $Q(T, x)$.}
    \label{small_tree_2}
\end{figure}
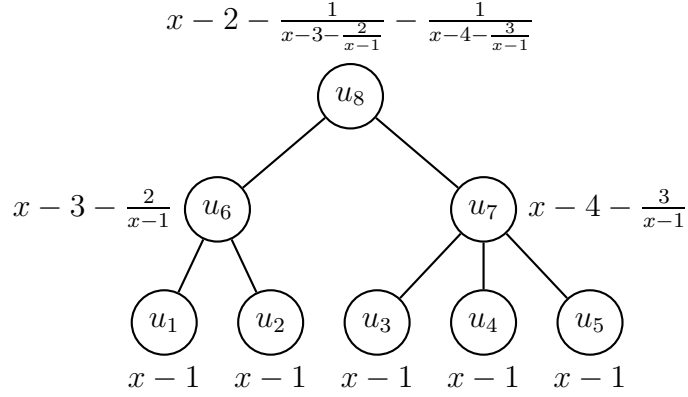
In this case we get
\begin{align*}
    \mathcal{H}(u_1, x) &= \mathcal{H}(u_2, x) = \mathcal{H}(u_3, x) = \mathcal{H}(u_4, x) = \mathcal{H}(u_5, x) = x - 1 ,\\
    \mathcal{H}(u_6, x) &= x - 3 - \frac{2}{x-1} = \dfrac{x^2 - 4x + 1}{x-1} ,\\
    \mathcal{H}(u_7, x) &= x - 4 - \frac{3}{x-1} = \dfrac{x^2 - 5x + 1}{x-1} ,\\
    \mathcal{H}(u_8, x) &= x - 2 - \frac{x-1}{x^2-4x + 1} - \frac{x-1}{x^2-5x+1} = \frac{x^5 - 11x^4 + 38x^3 - 42x^2 + 8x}{(x^2-4x+1)(x^2-5x+1)} ,
\end{align*}
which means that
\begin{align*}
    Q(T, x) &= \prod_{j = 1}^{8} \mathcal{H}(u_j, x)\\
    &= (x-1)^5 \ \dfrac{x^2 - 4x + 1}{x-1} \ \dfrac{x^2 - 5x + 1}{x-1} \ \frac{x^5 - 11x^4 + 38x^3 - 42x^2 + 8x}{(x^2-4x+1)(x^2-5x+1)}\\
    &= x(x-1)^3 (x^4 - 11x^3 + 38x^2 - 42x+8)\\
    &= x(x-1)^3 (x-4) (x^3 - 7x^2 + 10x - 2) .
\end{align*}
We have computed $Q(T, x)$ and it is now clear that the spectrum of $L(T)$ is composed of the simple eigenvalues $0$ and $4$, an eigenvalue $1$ of multiplicity three, as well as the roots of $x^3 - 7x^2 + 10x - 2$. It is trivial to numerically check that these roots are also simple eigenvalues.
\end{example}

\begin{example}\label{example_2}
    Let $T$ be the rooted tree given on Figure \ref{small_tree_3}.
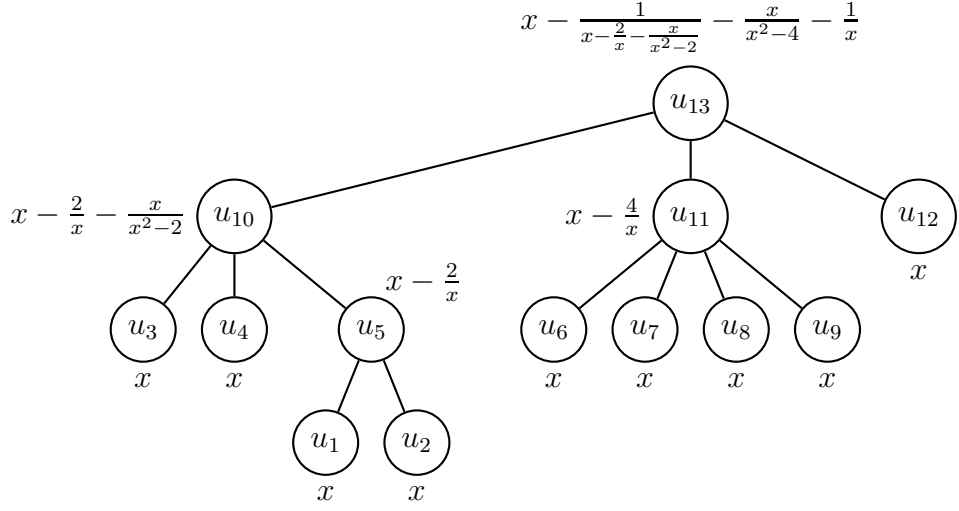
\begin{figure}
    \centering
    \begin{tikzpicture}
        \node[state, minimum size=0.25cm, thick, label=below:$x$] (12) at (2.4, -1.5) {$u_1$};
        \node[state, minimum size=0.25cm, thick, label=below:$x$] (13) at (3.6, -1.5) {$u_2$};
    
        \node[state, minimum size=0.25cm, thick, label=below:$x$] (1) at (5.4, 0) {$u_6$};
        \node[state, minimum size=0.25cm, thick, label=below:$x$] (2) at (6.6, 0) {$u_7$};
        \node[state, minimum size=0.25cm, thick, label=below:$x$] (3) at (7.8, 0) {$u_8$};
        \node[state, minimum size=0.25cm, thick, label=below:$x$] (4) at (9, 0) {$u_9$};
        \node[state, minimum size=0.25cm, thick, label=below:$x$] (5) at (0, 0) {$u_3$};
        \node[state, minimum size=0.25cm, thick, label=below:$x$] (6) at (1.2, 0) {$u_4$};
        \node[state, minimum size=0.25cm, thick, label={[label distance=-0.2cm]80:$x-\frac{2}{x}$}] (7) at (3, 0) {$u_5$};
        
        \node[state, minimum size=0.25cm, thick, label=left:$x-\frac{4}{x}$] (8) at (7.2, 1.5) {$u_{11}$};
        \node[state, minimum size=0.25cm, thick, label=left:$x-\frac{2}{x}-\frac{x}{x^2-2}$] (9) at (1.2, 1.5) {$u_{10}$};
        \node[state, minimum size=0.25cm, thick, label=below:$x$] (10) at (10.2, 1.5) {$u_{12}$};

        \node[state, minimum size=0.25cm, thick, label=above:$x-\frac{1}{x-\frac{2}{x}-\frac{x}{x^2-2}}-\frac{x}{x^2-4}-\frac{1}{x}$] (11) at (7.2, 3) {$u_{13}$};

        \path[thick] (12) edge (7);
        \path[thick] (13) edge (7);
        \path[thick] (1) edge (8);
        \path[thick] (2) edge (8);
        \path[thick] (3) edge (8);
        \path[thick] (4) edge (8);
        \path[thick] (5) edge (9);
        \path[thick] (6) edge (9);
        \path[thick] (7) edge (9);
        \path[thick] (8) edge (11);
        \path[thick] (9) edge (11);
        \path[thick] (10) edge (11);
    \end{tikzpicture}
    \caption{The rooted tree $T$ from Example \ref{example_2}, along with its assigned rational functions $\mathcal{G}$ used to compute $P(T, x)$.}
    \label{small_tree_3}
\end{figure}
Corollary \ref{adj_cor} gives us
\begin{align*}
    \mathcal{G}(u_1, x) &= \mathcal{G}(u_2, x) = \mathcal{G}(u_3, x) = \mathcal{G}(u_4, x) = x ,\\
    \mathcal{G}(u_5, x) &= x - \dfrac{2}{x} = \dfrac{x^2-2}{x} ,\\
    \mathcal{G}(u_6, x) &= \mathcal{G}(u_7, x) = \mathcal{G}(u_8, x) = \mathcal{G}(u_9, x) = x ,\\
    \mathcal{G}(u_{10}, x) &= x - \dfrac{2}{x} - \dfrac{x}{x^2-2} = \dfrac{(x^2 - 1)(x^2 - 4)}{x(x^2-2)} ,\\
    \mathcal{G}(u_{11}, x) &= x - \dfrac{4}{x} = \dfrac{x^2-4}{x} ,\\
    \mathcal{G}(u_{12}, x) &= x ,\\
    \mathcal{G}(u_{13}, x) &= x - \dfrac{x(x^2-2)}{(x^2-1)(x^2-4)} - \dfrac{x}{x^2-4} - \dfrac{1}{x} = \dfrac{x^6 - 8x^4 + 12x^2 - 4}{x(x^2-1)(x^2-4)} ,
\end{align*}
as shown on Figure \ref{small_tree_3}. Furthermore,
\begin{align*}
    P(T, x) &= \prod_{j = 1}^{13} \mathcal{G}(u_j, x)\\
    &= x^9 \ \dfrac{x^2-2}{x} \ \dfrac{(x^2-1)(x^2-4)}{x(x^2-2)} \ \dfrac{x^2-4}{x} \ \dfrac{x^6 - 8x^4 + 12x^2 - 4}{x(x^2-1)(x^2-4)}\\
    &= x^5 (x^2-4)(x^6 - 8x^4 + 12x^2 - 4) .
\end{align*}
We have computed $P(T, x)$, which allows us to determine the spectrum of $A(T)$. It is not difficult to see that it is composed of two simple eigenvalues $\pm 2$, an eigenvalue $0$ whose multiplicity is five, as well as six additional simple eigenvalues which represent the roots of the polynomial $x^6 - 8x^4 + 12x^2 - 4$. It is easy to numerically check that the polynomial $x^6 - 8x^4 + 12x^2 - 4$ really does have six distinct real roots which are all different from $0$ and $\pm 2$.

Finally, we will rely on Corollary \ref{lap_cor} in order to find $Q(T, x)$. The key part of the computation process is shown on Figure \ref{small_tree_4}.
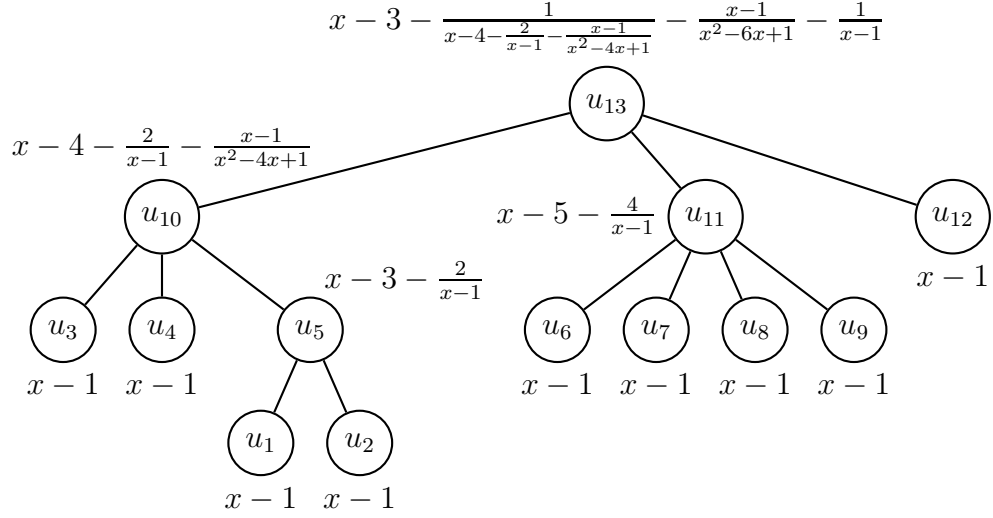
\begin{figure}
    \centering
    \begin{tikzpicture}
        \node[state, minimum size=0.25cm, thick, label=below:$x-1$] (12) at (2.6, -1.5) {$u_1$};
        \node[state, minimum size=0.25cm, thick, label=below:$x-1$] (13) at (3.9, -1.5) {$u_2$};
    
        \node[state, minimum size=0.25cm, thick, label=below:$x-1$] (1) at (6.5, 0) {$u_6$};
        \node[state, minimum size=0.25cm, thick, label=below:$x-1$] (2) at (7.8, 0) {$u_7$};
        \node[state, minimum size=0.25cm, thick, label=below:$x-1$] (3) at (9.1, 0) {$u_8$};
        \node[state, minimum size=0.25cm, thick, label=below:$x-1$] (4) at (10.4, 0) {$u_9$};
        \node[state, minimum size=0.25cm, thick, label=below:$x-1$] (5) at (0, 0) {$u_3$};
        \node[state, minimum size=0.25cm, thick, label=below:$x-1$] (6) at (1.3, 0) {$u_4$};
        \node[state, minimum size=0.25cm, thick, label={[label distance=-0.2cm]80:$x-3-\frac{2}{x-1}$}] (7) at (3.25, 0) {$u_5$};
        
        \node[state, minimum size=0.25cm, thick, label={left:$x-5-\frac{4}{x-1}$}] (8) at (8.45, 1.5) {$u_{11}$};
        \node[state, minimum size=0.25cm, thick, label=above:$x-4-\frac{2}{x-1}-\frac{x-1}{x^2-4x+1}$] (9) at (1.3, 1.5) {$u_{10}$};
        \node[state, minimum size=0.25cm, thick, label=below:$x-1$] (10) at (11.7, 1.5) {$u_{12}$};

        \node[state, minimum size=0.25cm, thick, label=above:$x-3-\frac{1}{x-4-\frac{2}{x-1}-\frac{x-1}{x^2-4x+1}}-\frac{x-1}{x^2-6x+1}-\frac{1}{x-1}$] (11) at (7.15, 3) {$u_{13}$};

        \path[thick] (12) edge (7);
        \path[thick] (13) edge (7);
        \path[thick] (1) edge (8);
        \path[thick] (2) edge (8);
        \path[thick] (3) edge (8);
        \path[thick] (4) edge (8);
        \path[thick] (5) edge (9);
        \path[thick] (6) edge (9);
        \path[thick] (7) edge (9);
        \path[thick] (8) edge (11);
        \path[thick] (9) edge (11);
        \path[thick] (10) edge (11);
    \end{tikzpicture}
    \caption{The rooted tree $T$ from Example \ref{example_2}, together with the assigned rational functions $\mathcal{H}$ which are used to compute $Q(T, x)$.}
    \label{small_tree_4}
\end{figure}
We compute
\begin{align*}
    \mathcal{H}(u_1, x) &= \mathcal{H}(u_2, x) = \mathcal{H}(u_3, x) = \mathcal{H}(u_4, x) = x-1 ,\\
    \mathcal{H}(u_5, x) &= x - 3 - \dfrac{2}{x-1} = \dfrac{x^2-4x+1}{x-1} ,\\
    \mathcal{H}(u_6, x) &= \mathcal{H}(u_7, x) = \mathcal{H}(u_8, x) = \mathcal{H}(u_9, x) = x-1 ,\\
    \mathcal{H}(u_{10}, x) &= x - 4 - \dfrac{2}{x-1} - \dfrac{x-1}{x^2-4x+1} = \dfrac{x^4 - 9x^3 + 22x^2 - 11x + 1}{(x-1)(x^2-4x+1)} ,\\
    \mathcal{H}(u_{11}, x) &= x - 5 - \dfrac{4}{x-1} = \dfrac{x^2-6x + 1}{x-1} ,\\
    \mathcal{H}(u_{12}, x) &= x - 1 ,\\
    \mathcal{H}(u_{13}, x) &= x - 3 - \dfrac{(x-1)(x^2-4x+1)}{x^4 - 9x^3 + 22x^2 - 11x + 1} - \dfrac{x-1}{x^2-6x + 1} - \dfrac{1}{x-1}\\
    &= \dfrac{x^8-19x^7+137x^6-467x^5+763x^4-541x^3+155x^2-13x}{(x-1)(x^2-6x+1)(x^4-9x^3+22x^2-11x+1)} .
\end{align*}
This gives us
\begin{align*}
    Q(T, x) &= \prod_{j = 1}^{13} \mathcal{H}(u_j, x)\\
    &= (x-1)^9 \ \dfrac{x^2-4x+1}{x-1} \ \dfrac{x^4 - 9x^3 + 22x^2 - 11x + 1}{(x-1)(x^2-4x+1)} \ \dfrac{x^2-6x + 1}{x-1}\\
    &\qquad\quad \cdot \dfrac{x^8-19x^7+137x^6-467x^5+763x^4-541x^3+155x^2-13x}{(x-1)(x^2-6x+1)(x^4-9x^3+22x^2-11x+1)}\\
    &= x (x-1)^5 (x^7-19x^6+137x^5-467x^4+763x^3-541x^2+155x-13) .
\end{align*}
If we need to determine the spectrum of $L(T)$ as well, we can immediately notice that it contains the simple eigenvalue $0$ and the eigenvalue $1$ whose multiplicity is five. The remaining eigenvaleus can be computed numerically without issues.
\end{example}

As we have seen in Examples \ref{example_1} and \ref{example_2}, Corollaries \ref{adj_cor} and \ref{lap_cor} can be implemented in order to compute the characteristic polynomial of both the adjacency and the Laplacian matrix of any rooted tree. By comparing these two basic examples, it becomes apparent that if the tree does not have a regular structure, then the manual computation process gets noticeably more tedious as the tree grows in height. The next section will implement these two corollaries on balanced trees, and thereby focus on the rooted trees which do have a regular structure.

\section{Spectral properties of balanced trees}\label{sc_balanced}
\subsection{General results}\label{subsc_balanced_main}

For a given balanced tree $T$, let $c_1, c_2, \ldots, c_{l(T)}$ be the sequence which defines how many children every vertex from each level has. In other words, $c_j$ should represent the number of children of every vertex from level $j$, for all the $1 \le j \le l(T)$. Here, it is important to note that $c_j = d(v) - 1$ holds for each $v \in V(T, j)$, provided $j > 1$. Also, it is clear that $c_1 = d(r(T))$. In order to resume our investigation of spectral properties regarding balanced trees, we will need the following short lemma.
\begin{lemma}\label{balanced_lemma}
    Let $T$ be any balanced tree with an arbitrarily chosen $\beta$-sequence such that any two vertices on the same level have equal corresponding $\beta$-elements. If we recursively assign a rational function $\mathcal{F}(v, x) \in \mathbb{Z}(x)$ to each vertex $v \in V(T)$ via Eq.\ (\ref{recursive}), then any two vertices on the same level must have equal assigned rational functions.
\end{lemma}
\begin{proof}
    The lemma is straightforward to prove via mathematical induction. The statement obviously holds for the vertices on the last level $l(T)$, given the fact that Eq.\ (\ref{recursive}) amounts to $\mathcal{F}(v, x) = x - \beta(v)$ for any $v \in V(T, l(T))$. Suppose that the statement holds for the vertices on level $j+1$, where $1 \le j \le l(T)-1$. We will now prove that it must hold for the vertices on level $j$ as well.
    
    Let $v$ be any vertex from level $j$. Also, let $v_1 \in V(T, j)$ and $v_2 \in V(T, j+1)$ be two arbitrarily chosen fixed vertices. By implementing the induction hypothesis, as well as the fact that $\beta(v) = \beta(v_1)$, Eq.\ (\ref{recursive}) helps us obtain
    \begin{align*}
        \mathcal{F}(v, x) &= x - \beta(v) - \sum_{w \in c(v)} \dfrac{1}{\mathcal{F}(w, x)}\\
        &= x - \beta(v_1) - \dfrac{c_j}{\mathcal{F}(v_2, x)} .
    \end{align*}
    It becomes clear that $\mathcal{F}(v, x)$ is the same for any chosen $v \in V(T, j)$, which completes the proof.
\end{proof}

The key observation is that Lemma \ref{balanced_lemma} can be implemented on the assigned rational functions $\mathcal{G}$ and $\mathcal{H}$ defined in Corollaries \ref{adj_cor} and \ref{lap_cor}, respectively. This implies that for any balanced tree $T$ it is much more convenient to view the aforementioned assigned rational functions as sequences $\mathcal{G}_1, \mathcal{G}_2, \ldots, \mathcal{G}_{l(T)}$ and $\mathcal{H}_1, \mathcal{H}_2, \ldots, \mathcal{H}_{l(T)}$, such that $\mathcal{G}_j$ and $\mathcal{H}_j$ are assigned to every vertex $v \in V(T, j)$, for all the $1 \le j \le l(T)$. This swiftly leads to the following two corollaries.

\begin{corollary}\label{meh_1}
    Let $T$ be an arbitrary balanced tree. If $\mathcal{G}_1, \mathcal{G}_2, \ldots, \mathcal{G}_{l(T)} \in \mathbb{Z}(x)$ is a sequence of rational functions which is recursively defined via
    \begin{align}
        \nonumber\mathcal{G}_{l(T)} &= x ,\\
        \label{cool_g}\mathcal{G}_j &= x - \dfrac{c_j}{\mathcal{G}_{j+1}} \qquad (\forall j \in \overline{1, l(T)-1}) ,
    \end{align}
    then
    \begin{align}\label{meh_formula_1}
        P(T, x) &= \prod_{j=1}^{l(T)} \mathcal{G}_j^{n(T, j)} .
    \end{align}
\end{corollary}
\begin{corollary}\label{meh_2}
    Let $T$ be any nontrivial balanced tree. If $\mathcal{H}_1, \mathcal{H}_2, \ldots, \mathcal{H}_{l(T)} \in \mathbb{Z}(x)$ is a sequence of rational functions which is recursively defined via
    \begin{align}
        \label{problematic_eq}\mathcal{H}_{l(T)} &= x - 1 ,\\
        \nonumber\mathcal{H}_j &= x - (c_j + 1) - \dfrac{c_j}{\mathcal{H}_{j+1}} \qquad (\forall j \in \overline{2, l(T)-1}) ,\\
        \nonumber\mathcal{H}_1 &= x - c_1 - \dfrac{c_1}{\mathcal{H}_2} ,
    \end{align}
    then
    \begin{align}\label{meh_formula_2}
        Q(T, x) &= \prod_{j=1}^{l(T)} \mathcal{H}_j^{n(T, j)} .
    \end{align}
\end{corollary}
\begin{remark}
    Corollary \ref{meh_2} does not hold if $T$ is the trivial balanced tree, so the condition that the tree is nontrivial cannot be omitted. This is due to the fact that the trivial tree is such that all of its leaves have no parent, which is the only case when Eq.~(\ref{problematic_eq}) does not hold.
\end{remark}

Corollaries \ref{meh_1} and \ref{meh_2} follow immediately from Corollaries \ref{adj_cor} and \ref{lap_cor} by implementing the elaborated consequences of Lemma \ref{balanced_lemma}. In fact, Eqs.\ (\ref{meh_formula_1}) and (\ref{meh_formula_2}) can be transformed into simpler forms which avoid rational functions and use only polynomials. This is demonstrated in the following theorem.

\begin{theorem}\label{second_th}
For an arbitrary balanced tree $T$, let $W_0, W_1, W_2, \ldots, W_{l(T)} \in \mathbb{Z}[x]$ be a sequence of polynomials defined via the recurrence relation
\begin{align}\label{second_th_rec}
\begin{split}
    W_0 &= 1 ,\\
    W_1 &= x ,\\
    W_j &= x W_{j-1} - c_{l(T)+1-j} W_{j-2} \qquad (\forall j \in \overline{2, l(T)}) .
\end{split}
\end{align}
We then have
\begin{align}
    \label{p_balanced}P(T, x) &= \prod_{j = 1}^{l(T)} W_j^{n(T, l(T) + 1 - j) - n(T, l(T) - j)} .
\end{align}
Also, if $T$ is nontrivial, then for the sequence $Y_0, Y_1, Y_2, \ldots, Y_{l(T)} \in \mathbb{Z}[x]$ of polynomials determined via
\begin{align*}
    Y_0 &= 1 ,\\
    Y_1 &= x - 1 ,\\
    Y_j &= (x - c_{l(T)+1-j} - 1) Y_{j-1} - c_{l(T)+1-j} Y_{j-2} \qquad (\forall j \in \overline{2, l(T)-1}) ,\\
    Y_{l(T)} &= (x - c_1) Y_{l(T)-1} - c_1 Y_{l(T)-2} ,
\end{align*}
we get
\begin{align}
    \label{q_balanced}Q(T, x) &= \prod_{j = 1}^{l(T)} Y_j^{n(T, l(T) + 1 - j) - n(T, l(T) - j)} .
\end{align}
\end{theorem}
\begin{proof}
    We will prove only Eq.\ (\ref{p_balanced}) due to the fact that Eq.\ (\ref{q_balanced}) is proved in an analogous manner by implementing Corollary \ref{meh_2} instead of Corollary \ref{meh_1}. First of all, we will show via mathematical induction that
    \begin{align}\label{g_quotient}
        \mathcal{G}_j = \dfrac{W_{l(T)+1-j}}{W_{l(T)-j}}
    \end{align}
    holds for all the $1 \le j \le l(T)$, where $\mathcal{G}_j$ is the sequence of assigned rational functions from Corollary \ref{meh_1}. It can immediately be seen that Eq.\ (\ref{g_quotient}) holds for the rational function assigned to the vertices on the last level, since we know that $\mathcal{G}_{l(T)} = x$, as well as $\dfrac{W_1}{W_0} = \dfrac{x}{1} = x$. Now suppose that Eq.\ (\ref{g_quotient}) is true for the rational function assigned to the vertices on level $j+1$, where $1 \le j \le l(T) - 1$. We will complete this part of the proof by showing that Eq.\ (\ref{g_quotient}) must necessarily be valid for the rational function assigned to the vertices on level $j$ as well.
    
    From Eq.\ (\ref{cool_g}) we quickly obtain
    \begin{align*}
        \mathcal{G}_j &= x - \dfrac{c_j}{\mathcal{G}_{j+1}}\\
        &= x - c_j \dfrac{W_{l(T)-j-1}}{W_{l(T)-j}} \\
        &= \dfrac{x W_{l(T)-j} - c_j W_{l(T)-j-1}}{W_{l(T)-j}}\\
        &= \dfrac{W_{l(T)+1-j}}{W_{l(T)-j}} ,
    \end{align*}
    which implies that Eq.\ (\ref{g_quotient}) holds for the rational function assigned to the vertices on level $j$, as needed.
    
    Taking into consideration Eq.\ (\ref{g_quotient}) along with Eq.\ (\ref{meh_formula_1}), we conclude that
    \begin{align*}
        P(T, x) &= \prod_{j=1}^{l(T)} \mathcal{G}_j^{n(T, j)} = \prod_{j=1}^{l(T)} \left( \dfrac{W_{l(T)+1-j}}{W_{l(T)-j}} \right)^{n(T, j)} = \dfrac{\displaystyle\prod_{j=1}^{l(T)} W_{l(T)+1-j}^{n(T, j)}}{\displaystyle\prod_{j=1}^{l(T)} W_{l(T)-j}^{n(T, j)}} .
    \end{align*}
    This means that
    \begin{align*}
        P(T, x) &= \dfrac{\displaystyle\prod_{j=1}^{l(T)} W_j^{n(T, l(T)+1-j)}}{\displaystyle\prod_{j=0}^{l(T)-1} W_j^{n(T, l(T)-j)}} .
    \end{align*}
    By using $W_0 = 1$, we further get
    \begin{align*}
        P(T, x) &= \dfrac{\displaystyle\prod_{j=1}^{l(T)} W_j^{n(T, l(T)+1-j)}}{W_0^{n(T, l(T))}\displaystyle\prod_{j=1}^{l(T)-1} W_j^{n(T, l(T)-j)}} = \dfrac{\displaystyle\prod_{j=1}^{l(T)} W_j^{n(T, l(T)+1-j)}}{\displaystyle\prod_{j=1}^{l(T)-1} W_j^{n(T, l(T)-j)}} .
    \end{align*}
    Finally, since $n(T, 0) = 0$, we reach
    \begin{align*}
        P(T, x) &= \dfrac{\displaystyle\prod_{j=1}^{l(T)} W_j^{n(T, l(T)+1-j)}}{W_{l(T)}^{n(T, 0)}\displaystyle\prod_{j=1}^{l(T)-1} W_j^{n(T, l(T)-j)}} = \dfrac{\displaystyle\prod_{j=1}^{l(T)} W_j^{n(T, l(T)+1-j)}}{\displaystyle\prod_{j=1}^{l(T)} W_j^{n(T, l(T)-j)}}\\
        &= \prod_{j=1}^{l(T)} W_j^{n(T, l(T)+1-j) - n(T, l(T)-j)} .
    \end{align*}
\end{proof}

Theorem \ref{second_th} has some interesting direct consequences. For example, Eq.\ (\ref{p_balanced}) makes it easy to determine $\sigma^{*}(T)$, as shown in the next corollary.

\begin{corollary}\label{phi_cor}
    Let $T$ be an arbitrary balanced tree and let $\Phi \subseteq \mathbb{N}$ be the set
    \begin{align*}
        \Phi &= \{ l(T) \} \cup \{ j \in \mathbb{N} \colon 1 \le j \le l(T) - 1, c_{l(T)-j} > 1 \} .
    \end{align*}
    Then
    \begin{align}
        \label{phi_formula}\sigma^{*}(T) &= \bigcup_{j \in \Phi} \{ x \in \mathbb{R} \colon W_j(x) = 0 \} .
    \end{align}
\end{corollary}
\begin{proof}
    From Eq.\ (\ref{p_balanced}) we have that $P(T, x)$ can be represented as a product of certain elements of the $W_j$ sequence. Moreover, some $W_j$ appears as a factor of $P(T, x)$ if and only if
    \begin{alignat}{2}
        \nonumber && n(T, l(T)+1-j) - n(T, l(T)-j) &> 0\\
        \label{phi_condition} \iff \quad && n(T, l(T)+1-j) &> n(T, l(T)-j) .
    \end{alignat}
    For $j = l(T)$, this condition is always satisfied, due to the fact that $n(T, 1) = 1$ and $n(T, 0) = 0$. If $1 \le j \le l(T) - 1$, we then have
    \begin{align*}
        n(T, l(T) + 1 - j) = c_{l(T) - j} n(T, l(T) - j) ,
    \end{align*}
    which means that Eq.\ (\ref{phi_condition}) is equivalent to $c_{l(T)-j} > 1$. Thus, we get that $W_j$ appears as a factor of $P(T, x)$ if and only if $j \in \Phi$. This clearly means that some real number belongs to $\sigma^{*}(T) \subseteq \mathbb{R}$ if and only if it is a root of at least one of the polynomials $W_j$ where $j \in \Phi$. The noted observation promptly leads to Eq.~(\ref{phi_formula}).
\end{proof}

\begin{remark}
    It is worth noting that Eq.\ (\ref{q_balanced}) could be used in order to make a similar conclusion regarding the set of distinct eigenvalues of $L(T)$ and the $Y_j$ sequence of polynomials.
\end{remark}

\subsection{Spectral properties of Bethe trees}\label{subsc_bethe}

In this subsection we will implement Theorem \ref{second_th} and Corollary \ref{phi_cor} on the Bethe tree $\mathcal{B}_{d, k}$ in order to compute its characteristic polynomial $P(\mathcal{B}_{d, k}, x)$, set of distinct eigenvalues $\sigma^{*}(\mathcal{B}_{d, k})$ and energy $\mathcal{E}(\mathcal{B}_{d, k})$, for all the $d \ge 2$ and $k \ge 1$. Let $E_0(x, a), E_1(x, a), E_2(x, a), \ldots$ be the sequence of polynomials defined via the recurrence relation
\begin{align}\label{dickson_rec}
\begin{split}
    E_0(x, a) &= 1 ,\\
    E_1(x, a) &= x ,\\
    E_j(x, a) &= x E_{j-1}(x, a) - a E_{j-2}(x, a) \qquad (\forall j \ge 2) ,
\end{split}
\end{align}
where $a \in \mathbb{R}$ is a fixed constant. We shall call these polynomials the Dickson polynomials of the second kind, as done so in \cite[pp.\ 9--10]{dickson_polynomials}. There are many known properties regarding the polynomials from this sequence, which is very convenient for us, given the fact that
\begin{align}\label{dickson_realization}
    W_j = E_j(x, d-1) \qquad (\forall j = \overline{0, l(T)}) ,
\end{align}
where $W_j$ is the corresponding sequence from Theorem \ref{second_th} when it is applied on the balanced tree $T = \mathcal{B}_{d, k}$. This is not difficult to see, since $T = \mathcal{B}_{d, k}$ implies $c_j = d-1$ for all the $1 \le j \le l(T) - 1$, which immediately makes the recurrence relation Eq.~(\ref{second_th_rec}) defining $W_0, W_1, \ldots, W_{l(T)}$ equivalent to the recurrence relation Eq.~(\ref{dickson_rec}). This observation leads to the following theorem.

\begin{theorem}\label{bethe_th}
    For any Bethe tree $\mathcal{B}_{d, k}$, where $d \ge 2$ and $k \ge 1$, we have
    \begin{align}\label{bethe_p}
        P(\mathcal{B}_{d, k}, x) &= E_k(x, d-1) \prod_{j = 1}^{k-1} E_j(x, d-1)^{(d-2)(d-1)^{k-1-j}} .
    \end{align}
\end{theorem}
\begin{proof}
    If we compare Eq.\ (\ref{bethe_p}) to Eq.\ (\ref{p_balanced}) and use Eq.\ (\ref{dickson_realization}), as well as $k = l(T)$, it becomes sufficient to show that $n(T, 1) - n(T, 0) = 1$ and
    \begin{align*}
        n(T, k + 1 - j) - n(T, k - j) &= (d-2)(d-1)^{k-1-j} ,
    \end{align*}
    for all of the $1 \le j \le k-1$. The first fact is obvious, while the second follows directly from $n(T, k+1-j) = (d-1)^{k-j}$ and $n(T, k-j) = (d-1)^{k-1-j}$.
\end{proof}

We can now use Corollary \ref{phi_cor} together with some known properties about the Dickson polynomials of the second kind in order to determine $\sigma^{*}(\mathcal{B}_{d, k})$. The full procedure is given in the next corollary.

\begin{corollary}
    For any Bethe tree $\mathcal{B}_{d, k}$, where $d \ge 3$ and $k \ge 1$, we have
    \begin{align}\label{sigma_bethe_3}
        \sigma^{*}(\mathcal{B}_{d, k}) &= \left\{ 2\sqrt{d-1} \cos\left(\frac{h}{j+1}\pi\right) \colon 1 \le h \le j \le k \right\} .
    \end{align}
    Also, for any $k \ge 1$, we have
    \begin{align}\label{sigma_bethe_2}
        \sigma^{*}(\mathcal{B}_{2, k}) &= \left\{ 2\cos\left(\frac{h}{k+1}\pi\right) \colon 1 \le h \le k \right\} .
    \end{align}
\end{corollary}
\begin{proof}
    The Bethe tree $\mathcal{B}_{2, k}$ represents a path graph composed of $k$ vertices. It is known (see, for example, \cite[pp.\ 18]{spectra_of_graphs}) that the spectrum of this graph is composed of the simple eigenvalues $2\cos\left(\frac{1}{k+1}\pi\right), 2\cos\left(\frac{2}{k+1}\pi\right), \ldots,  2\cos\left(\frac{k}{k+1}\pi\right)$, which proves Eq.\ (\ref{sigma_bethe_2}).
    
    On the other hand, it is also known that $E_j(x, a)$ has $j$ distinct simple roots \linebreak $2\sqrt{a}\cos\left(\frac{1}{j+1}\pi\right), 2\sqrt{a}\cos\left(\frac{2}{j+1}\pi\right), \ldots, 2\sqrt{a}\cos\left(\frac{j}{j+1}\pi\right)$ (see, for example, \cite[pp.\ \linebreak 9--10]{dickson_polynomials}). By implementing Corollary \ref{phi_cor}, we get that $\Phi = \{1, 2, \ldots, k \}$, since $c_j > 1$ for all of the $1 \le j \le k - 1$, provided $d \ge 3$. This means that for $d \ge 3 $, we have that $\sigma^{*}(\mathcal{B}_{d, k})$ is composed of the real numbers which are a root of at least one polynomial from the sequence $E_1(x, d-1), E_2(x, d-1), \ldots, E_k(x, d-1)$. However, the set of such real numbers is obviously equal to the set represented in Eq.~(\ref{sigma_bethe_3}).
\end{proof}

To finish our investigation of spectral properties regarding the Bethe trees, we shall determine the energy of $\mathcal{B}_{d, k}$. Our key result regarding this matter is presented in the following theorem.

\begin{theorem}\label{bethe_main}
    For an arbitrary Bethe tree $\mathcal{B}_{d, k}$, where $d \ge 3$ and $k \ge 1$, we have
    \begin{align}\label{bethe_energy_1}
        \mathcal{E}(\mathcal{B}_{d, k}) &= \sum_{j = 1}^{k-1} f_j (d-1)^{k-\frac{1}{2}-j} ,
    \end{align}
    where
    \begin{align*}
        f_j &= \begin{cases}
        2 \csc\left( \dfrac{\pi}{2j+4} \right) - 2\cot\left( \dfrac{\pi}{2j+2} \right), & 2 \nmid j ,\\
        2 \cot\left( \dfrac{\pi}{2j+4} \right) - 2\csc\left( \dfrac{\pi}{2j+2} \right), & 2 \mid j .
    \end{cases}
    \end{align*}
    Also, for any $k \ge 1$, we have
    \begin{align}\label{bethe_energy_2}
        \mathcal{E}(\mathcal{B}_{2, k}) &= \begin{cases}
            2 \left( \cot\left( \dfrac{\pi}{2k+2} \right) - 1 \right) , & 2 \nmid k ,\\
            2 \left( \csc\left( \dfrac{\pi}{2k+2} \right) - 1 \right) , & 2 \mid k .
        \end{cases}
    \end{align}
\end{theorem}

In order to provide a proof of Theorem \ref{bethe_main}, we will rely on certain properties regarding the $E_j(x, a)$ polynomials. Let $\Psi(E_j(x, a))$ denote the sum of absolute values of all of the roots of $E_j(x, a)$. The upcoming auxiliary lemma solves the problem of giving an explicit formula for $\Psi(E_j(x, a))$.

\begin{lemma}\label{psi_lemma}
    For any Dickson polynomial of the second kind $E_j(x, a)$, we have
    \begin{align*}
        \Psi(E_j(x, a)) &= \begin{cases}
            2 \sqrt{a} \left( \cot\left( \dfrac{\pi}{2j+2} \right) - 1 \right) , & 2 \nmid j ,\\
            2 \sqrt{a} \left( \csc\left( \dfrac{\pi}{2j+2} \right) - 1 \right) , & 2 \mid j .
        \end{cases}
    \end{align*}
\end{lemma}
\begin{proof}
    As discussed earlier, we know that the roots of the polynomial $E_j(x, a)$ are $2\sqrt{a}\cos\left(\frac{1}{j+1}\pi\right), 2\sqrt{a}\cos\left(\frac{2}{j+1}\pi\right), \ldots, 2\sqrt{a}\cos\left(\frac{j}{j+1}\pi\right)$. Hence
    \begin{align*}
        \Psi(E_j(x, a)) &= \sum_{h = 1}^{j}\left| 2\sqrt{a} \cos\left(\frac{h}{j+1}\pi\right) \right| = 2\sqrt{a} \sum_{h = 1}^{j}\left|\cos\left(\frac{h}{j+1}\pi\right) \right| .
    \end{align*}
    Since $\cos\left(\frac{h}{j+1}\pi\right) > 0$ for $1 \le h < \frac{j+1}{2}$ and $\cos\left(\frac{h}{j+1}\pi\right) = -\cos\left(\frac{j+1-h}{j+1}\pi\right)$ for all the $\frac{j+1}{2} < h \le j$, we can rewrite the last expression as
    \begin{align*}
        \Psi(E_j(x, a)) &= 4\sqrt{a} \sum_{h = 1}^{\lfloor\frac{j}{2}\rfloor}\cos\left(\frac{h}{j+1}\pi\right) .
    \end{align*}
    Let us denote $\zeta = e^{\frac{i \pi}{j+1}}$. It is convenient to replace $\cos\left(\frac{h}{j+1}\pi\right)$ with $\frac{\zeta^h + \zeta^{-h}}{2}$. This gives us:
    \begin{align*}
        \Psi(E_j(x, a)) &= 4\sqrt{a} \sum_{h = 1}^{\lfloor\frac{j}{2}\rfloor} \frac{\zeta^h + \zeta^{-h}}{2}\\
        &= 2\sqrt{a} \sum_{h = 1}^{\lfloor\frac{j}{2}\rfloor} \left(\zeta^h + \zeta^{-h}\right)\\
        &= 2\sqrt{a} \left( \left( \sum_{h = -\lfloor\frac{j}{2}\rfloor}^{\lfloor\frac{j}{2}\rfloor}\zeta^h \right) - 1 \right) .
    \end{align*}
    Since $\zeta \neq 1$, we can use the standard formula for summing a geometric progression in order to get
    \begin{align*}
        \Psi(E_j(x, a)) &= 2\sqrt{a} \left( \dfrac{\sum_{h = 0}^{2 \lfloor\frac{j}{2}\rfloor}\zeta^h}{\zeta^{\lfloor\frac{j}{2}\rfloor}} - 1 \right)\\
        &= 2\sqrt{a} \left( \dfrac{\zeta^{2 \lfloor\frac{j}{2}\rfloor + 1} - 1}{\zeta^{\lfloor\frac{j}{2}\rfloor}(\zeta-1)} - 1 \right)\\
        &= 2\sqrt{a} \left( \dfrac{\zeta^{\lfloor\frac{j}{2}\rfloor + 1} - \zeta^{-\lfloor\frac{j}{2}\rfloor}}{\zeta-1} - 1 \right)\\
        &= 2\sqrt{a} \left( \dfrac{\left(\zeta^{\lfloor\frac{j}{2}\rfloor + 1} - \zeta^{-\lfloor\frac{j}{2}\rfloor} \right)\left(\frac{1}{\zeta}-1\right)}{(\zeta-1)\left(\frac{1}{\zeta}-1\right)} -1 \right) .
    \end{align*}
    By taking into consideration that
    \begin{align*}
        \left(\zeta^{\lfloor\frac{j}{2}\rfloor + 1} - \zeta^{-\lfloor\frac{j}{2}\rfloor} \right)\left(\frac{1}{\zeta}-1\right) &= \zeta^{\lfloor\frac{j}{2}\rfloor}+\zeta^{-\lfloor\frac{j}{2}\rfloor} - \zeta^{\lfloor\frac{j}{2}\rfloor+1} - \zeta^{-\lfloor\frac{j}{2}\rfloor-1}\\
        &= 2\cos\left( \frac{\lfloor\frac{j}{2}\rfloor}{j+1}\pi\right)-2\cos\left( \frac{\lfloor\frac{j}{2}\rfloor+1}{j+1}\pi\right)
    \end{align*}
    and
    \begin{align*}
        (\zeta-1)\left(\frac{1}{\zeta}-1\right) = 2 - \left(\zeta + \frac{1}{\zeta}\right) = 2 - 2 \cos\left(\frac{1}{j+1}\pi\right) ,
    \end{align*}
    we conclude that
    \begin{align}
        \nonumber\Psi(E_j(x, a)) &= 2\sqrt{a} \left( \frac{2\cos\left( \frac{\lfloor\frac{j}{2}\rfloor}{j+1}\pi\right)-2\cos\left( \frac{\lfloor\frac{j}{2}\rfloor+1}{j+1}\pi\right)}{2 - 2 \cos\left(\frac{1}{j+1}\pi\right)}-1 \right)\\
        \label{psi_derivation} &= 2\sqrt{a} \left( \frac{\cos\left( \frac{\lfloor\frac{j}{2}\rfloor}{j+1}\pi\right)-\cos\left( \frac{\lfloor\frac{j}{2}\rfloor+1}{j+1}\pi\right)}{1 - \cos\left(\frac{1}{j+1}\pi\right)}-1 \right) .
    \end{align}
    If $j$ is odd, then $\lfloor\frac{j}{2}\rfloor = \frac{j-1}{2}$ and $\frac{\lfloor\frac{j}{2}\rfloor+1}{j+1}\pi=\frac{\pi}{2}$, which transforms Eq.\ (\ref{psi_derivation}) into
    \begin{align*}
        \Psi(E_j(x, a)) &= 2\sqrt{a} \left( \frac{\cos\left( \frac{j-1}{j+1}\cdot\frac{\pi}{2}\right)}{1 - \cos\left(\frac{1}{j+1}\pi\right)}-1 \right)\\
        &= 2\sqrt{a} \left( \frac{\sin\left( \frac{2}{j+1}\cdot\frac{\pi}{2}\right)}{2 \sin^2\left(\frac{\pi}{2j+2}\right)}-1 \right)\\
        &= 2\sqrt{a} \left( \frac{\sin\left( \frac{\pi}{j+1}\right)}{2 \sin^2\left(\frac{\pi}{2j+2}\right)}-1 \right)\\
        &= 2\sqrt{a} \left( \frac{2\sin\left( \frac{\pi}{2j+2}\right)\cos\left( \frac{\pi}{2j+2}\right)}{2 \sin^2\left(\frac{\pi}{2j+2}\right)}-1 \right)\\
        &= 2\sqrt{a} \left( \cot\left( \frac{\pi}{2j+2} \right) - 1 \right) .
    \end{align*}
    If $j$ is even, then $\lfloor\frac{j}{2}\rfloor = \frac{j}{2}$, as well as $\cos\left( \frac{\lfloor\frac{j}{2}\rfloor+1}{j+1}\pi\right) = - \cos\left( \frac{\lfloor\frac{j}{2}\rfloor}{j+1}\pi\right)$, which gives
    \begin{align*}
        \Psi(E_j(x, a)) &= 2\sqrt{a} \left( \frac{2\cos\left( \frac{j}{j+1} \cdot \frac{\pi}{2}\right)}{1 - \cos\left(\frac{1}{j+1}\pi\right)}-1 \right)\\
        &= 2\sqrt{a} \left( \frac{2\sin\left( \frac{1}{j+1} \cdot \frac{\pi}{2}\right)}{2 \sin^2\left(\frac{\pi}{2j+2}\right)}-1 \right)\\
        &= 2\sqrt{a} \left( \csc\left( \frac{\pi}{2j+2} \right)-1 \right) .
    \end{align*}
\end{proof}

We are now able to implement Lemma \ref{psi_lemma} in order to finish the computation of $\mathcal{E}(\mathcal{B}_{d, k})$.

\bigskip
\noindent
{\em Proof of Theorem \ref{bethe_main}}.\quad First of all, from Eq.\ (\ref{bethe_p}) it is clear that $P(\mathcal{B}_{2, k}, x) = E_k(x, 1)$ for every $k \ge 1$. This means that $\mathcal{E}(\mathcal{B}_{2, k}) = \Psi(E_k(x, 1))$. Hence, Eq.\ (\ref{bethe_energy_2}) follows immediately from Lemma \ref{psi_lemma}.

Now, we will suppose that $d \ge 3$ and prove Eq.\ (\ref{bethe_energy_1}). From Eq.\ (\ref{bethe_p}), we obtain
\begin{align*}
    \mathcal{E}(\mathcal{B}_{d, k}) &= \Psi(E_k(x, d-1)) + \sum_{j = 1}^{k-1} (d-2)(d-1)^{k-1-j} \Psi(E_j(x, d-1)) .
\end{align*}
However,
\begin{align*}
    \sum_{j = 1}^{k-1} &(d-2)(d-1)^{k-1-j} \Psi(E_j(x, d-1)) = \\
    &= \sum_{j = 1}^{k-1} ((d-1)^{k-j} - (d-1)^{k-1-j}) \Psi(E_j(x, d-1))\\
    &= \sum_{j = 1}^{k-1} (d-1)^{k-j} \Psi(E_j(x, d-1)) - \sum_{j = 1}^{k-1} (d-1)^{k-1-j} \Psi(E_j(x, d-1))\\
    &= \sum_{j = 0}^{k-2} (d-1)^{k-1-j} \Psi(E_{j+1}(x, d-1)) - \sum_{j = 1}^{k-1} (d-1)^{k-1-j} \Psi(E_j(x, d-1)) .
\end{align*}
Given the fact that
\begin{align*}
    \Psi(E_k(x, d-1)) &+ \sum_{j = 0}^{k-2} (d-1)^{k-1-j} \Psi(E_{j+1}(x, d-1)) =\\
    &= \sum_{j = 0}^{k-1} (d-1)^{k-1-j} \Psi(E_{j+1}(x, d-1)) ,
\end{align*}
we further obtain
\begin{align*}
    \mathcal{E}(\mathcal{B}_{d, k}) &= \sum_{j = 0}^{k-1} (d-1)^{k-1-j} \Psi(E_{j+1}(x, d-1)) - \sum_{j = 1}^{k-1} (d-1)^{k-1-j} \Psi(E_j(x, d-1)) .
\end{align*}
Also, we know that $E_1(x, d-1) = x$, hence $\Psi(E_1(x, d-1)) = 0$, which gives us
\begin{align}
    \nonumber\mathcal{E}(\mathcal{B}_{d, k}) &= \sum_{j = 1}^{k-1} (d-1)^{k-1-j} \Psi(E_{j+1}(x, d-1)) - \sum_{j = 1}^{k-1} (d-1)^{k-1-j} \Psi(E_j(x, d-1))\\
    \label{energy_derivation} &= \sum_{j = 1}^{k-1} (d-1)^{k-1-j} \left( \Psi(E_{j+1}(x, d-1)) - \Psi(E_j(x, d-1)) \right) .
\end{align}
Taking into consideration Eq.\ (\ref{energy_derivation}), it becomes apparent that in order to prove Eq.~(\ref{bethe_energy_1}), it is sufficient to show that 
\begin{align*}
    \Psi(E_{j+1}(x, d-1)) - \Psi(E_j(x, d-1)) &= f_j \, \sqrt{d-1} \, ,
\end{align*}
for all the $1 \le j \le k-1$. However, this is straightforward to do with the help of Lemma \ref{psi_lemma}. If $j$ is odd, then
\begin{align*}
    \Psi(&E_{j+1}(x, d-1)) - \Psi(E_j(x, d-1)) = \\
    &= 2 \sqrt{d-1} \left( \csc\left( \dfrac{\pi}{2j+4} \right) - 1 \right) - 2 \sqrt{d-1} \left( \cot\left( \dfrac{\pi}{2j+2} \right) - 1 \right) \\
    &= 2 \sqrt{d-1} \left( \csc\left( \dfrac{\pi}{2j+4} \right) - \cot\left( \dfrac{\pi}{2j+2} \right) \right)\\
    &= f_j \, \sqrt{d-1} \, .
\end{align*}
On the other hand, if $j$ is even, we get
\begin{align*}
    \Psi(&E_{j+1}(x, d-1)) - \Psi(E_j(x, d-1)) = \\
    &= 2 \sqrt{d-1} \left( \cot\left( \dfrac{\pi}{2j+4} \right) - 1 \right) - 2 \sqrt{d-1} \left( \csc\left( \dfrac{\pi}{2j+2} \right) - 1 \right) \\
    &= 2 \sqrt{d-1} \left( \cot\left( \dfrac{\pi}{2j+4} \right) - \csc\left( \dfrac{\pi}{2j+2} \right) \right)\\
    &= f_j \, \sqrt{d-1} \, .
\end{align*}
\qed

\subsection{Spectral properties of anti-factorial trees}\label{subsc_antifactorial}

The anti-factorial trees represent another class of rooted trees which have a highly regular structure, as seen by their definition in Section \ref{sc_introduction}. This makes it possible to implement Theorem \ref{second_th} and Corollary \ref{phi_cor} in order to determine their characteric polynomial $P(\mathcal{A}_k, x)$ and set of distinct eigenvalues $\sigma^{*}(\mathcal{A}_k)$, in a similar manner as done so in the previous subsection. Let $H_0(x), H_1(x), H_2(x), \ldots$ be the Hermite polynomials, which represent a classical orthogonal polynomial sequence. It is known (see, for example, \cite[pp.\ 105--106]{szego}) that one of the ways to define this sequence is by using the recurrence relation
\begin{align*}
    H_0(x) &= 1 ,\\
    H_1(x) &= 2x ,\\
    H_j(x) &= 2x H_{j-1}(x) - 2(j-1) H_{j-2}(x) \qquad (\forall j \ge 2) .
\end{align*}
Here, it is important to note that by denoting $He_j(x) = 2^{- \frac{j}{2}} H_j\left(\frac{x}{\sqrt{2}}\right)$ for all of the $j \in \mathbb{N}_0$, we obtain an orthogonal polynomial sequence $He_0(x), He_1(x), He_2(x), \ldots$ whose members are also sometimes called the Hermite polynomials. It is not difficult to see that this sequence can alternatively be defined via the recurrence relation
\begin{align}\label{hermite_rec}
\begin{split}
    He_0(x) &= 1 ,\\
    He_1(x) &= x ,\\
    He_j(x) &= x He_{j-1}(x) - (j-1) He_{j-2}(x) \qquad (\forall j \ge 2) .
\end{split}
\end{align}

If we apply Theorem \ref{second_th} on the balanced tree $T = \mathcal{A}_k$, we get that the recurrence relation Eq.~(\ref{second_th_rec}) defining $W_0, W_1, \ldots, W_{l(T)}$ becomes equivalent to the recurrence relation Eq.~(\ref{hermite_rec}), due to the fact that $c_{l(T)+1-j} = k - (k+1-j) = j-1$, as needed. This observation directly gives us
\begin{align}\label{hermite_realization}
    W_j = He_j(x) \qquad (\forall j = \overline{0, l(T)}) ,
\end{align}
We are now in the position to quickly prove the following theorem.

\begin{theorem}
    For any anti-factorial tree $\mathcal{A}_k$, where $k \ge 1$, we have
    \begin{align}\label{antifactorial_p}
        P(\mathcal{A}_k, x) &= He_k(x) \prod_{j = 2}^{k-1} He_j(x)^{\frac{(j-1)(k-1)!}{j!}} .
    \end{align}
\end{theorem}
\begin{proof}
    By comparing Eq.\ (\ref{antifactorial_p}) to Eq.\ (\ref{p_balanced}), as well as using Eq.\ (\ref{hermite_realization}) together with $k = l(T)$, we conclude that proving the theorem statement gets down to showing that $n(T, 1) - n(T, 0) = 1$ and
    \begin{align*}
        n(T, k + 1 - j) - n(T, k - j) &= \frac{(j-1)(k-1)!}{j!} ,
    \end{align*}
    for all the $1 \le j \le k-1$. The former equality is clear, while the latter is easy to prove via simple mathematical calculation, if we take into consideration that
    \begin{align*}
        n(T, j) &= (k-1)(k-2) \cdots (k + 1 - j) = \dfrac{(k-1)!}{(k-j)!} ,
    \end{align*}
    for each $1 \le j \le k$. From here, we further obtain
    \begin{align*}
        n(T, k + 1 - j) - n(T, k - j) &= \dfrac{(k-1)!}{(j-1)!} - \dfrac{(k-1)!}{j!}\\
        &= j \dfrac{(k-1)!}{j!} - \dfrac{(k-1)!}{j!}\\
        &= \dfrac{(j-1)(k-1)!}{(k-j)!} ,
    \end{align*}
    for all the $1 \le j \le k-1$, which completes the proof.
\end{proof}

Similarly as done so in the previous subsection, we can apply Corollary \ref{phi_cor} in order to determine $\sigma^{*}(\mathcal{A}_k)$, as shown in the following corollary.

\begin{corollary}
    For any anti-factorial tree $\mathcal{A}_k$, where $k \ge 2$, we have
    \begin{align}\label{sigma_antifact_2}
        \sigma^{*}(\mathcal{A}_k) &= \bigcup_{j=2}^{k} \{ x \in \mathbb{R} \colon He_j(x) = 0 \} .
    \end{align}
    Also, we have
    \begin{align}\label{sigma_antifact_1}
        \sigma^{*}(\mathcal{A}_1) &= \{ 0 \} .
    \end{align}
\end{corollary}
\begin{proof}
    Eq.\ (\ref{sigma_antifact_1}) is obtained immediately by using the fact that the anti-factorial tree $\mathcal{A}_1$ is actually just a trivial tree. On the other hand, Eq.\ (\ref{sigma_antifact_2}) follows directly by applying Corollary \ref{phi_cor}, given the fact that $c_j > 1$ for all the $1 \le j \le k-2$, while $c_{k-1} = 1$, provided $k \ge 2$. This implies that whenever $k \ge 2$, we have $\Phi = \{ 2, 3, \ldots, k-1, k \}$. Therefore, Eq.~(\ref{phi_formula}) and Eq.\ (\ref{hermite_realization}) together imply Eq.~(\ref{sigma_antifact_2}).
\end{proof}

\section{Tree merging procedure}\label{sc_merging}

In this final section, we will demonstrate a tree merging technique which preserves the spectra of all of its input trees. To be more precise, this procedure takes a finite sequence of rooted trees $T_1, T_2, \ldots, T_k$, and outputs a newly formed rooted tree $T_0$ such that its spectrum $\sigma(T_0)$ satisfies
\begin{align}\label{preserving_formula}
    \sigma(T_0) \supseteq \bigcup_{j = 1}^{k} \sigma(T_j) .
\end{align}
If we incorporate the notation $\mu(T, \lambda)$ to denote the multiplicity of some real number $\lambda \in \mathbb{R}$ as an eigenvalue of a tree $T$, we conclude that the condition given in Eq.\ (\ref{preserving_formula}) swiftly becomes equivalent to
\begin{align}\label{preserving_formula_2}
    \mu(T_0, \lambda) \ge \sum_{j=1}^{k} \mu(T_j, \lambda) \qquad (\forall \lambda \in \mathbb{R}) .
\end{align}
It is interesting to note that not only does a tree construction method which satisfies Eq.\ (\ref{preserving_formula}) exist, but there exists one whose steps can be fairly easily elaborated. Our key result regarding this matter is presented in the following theorem.

\begin{theorem}\label{tree_construction_th}
    Let $T_1, T_2, \ldots, T_k$ be an arbitrary sequence of rooted trees. Also, let $\alpha_1, \alpha_2, \ldots, \alpha_k \in \mathbb{N}$ be any sequence of positive integers. If we use $T_0$ to denote a rooted tree such that
    \begin{itemize}
        \item the root has exactly $\displaystyle\sum_{j=1}^{k} \alpha_j$ children;
        \item the rooted subtrees corresponding to the first $\alpha_1$ root children are all isomorphic to $T_1$, then the rooted subtrees corresponding to the following $\alpha_2$ root children are all isomorphic to $T_2$, and so on;
    \end{itemize}
    we then have
    \begin{align*}
        \mu(T_0, \lambda) \ge \sum_{j=1}^{k} (\alpha_j - 1) \mu(T_j, \lambda) \qquad (\forall \lambda \in \mathbb{R}) .
    \end{align*}
\end{theorem}

If we apply Theorem \ref{tree_construction_th} on the sequence $\alpha_1 = \alpha_2 = \cdots = \alpha_k = 2$, we obtain the following corollary.
\begin{corollary}\label{tree_construction_cor}
    Let $T_1, T_2, \ldots, T_k$ be an arbitrary sequence of rooted trees. If we use $T_0$ to denote a rooted tree such that
    \begin{itemize}
        \item the root has exactly $2k$ children;
        \item the rooted subtrees corresponding to the first two root children are both isomorphic to $T_1$, then the rooted subtrees corresponding to the following two root children are both isomorphic to $T_2$, and so on;
    \end{itemize}
    we then have
    \begin{align*}
        \mu(T_0, \lambda) \ge \sum_{j=1}^{k} \mu(T_j, \lambda) \qquad (\forall \lambda \in \mathbb{R}) .
    \end{align*}
\end{corollary}
Due to the equivalence of Eq.\ (\ref{preserving_formula_2}) and Eq.\ (\ref{preserving_formula}), Corollary \ref{tree_construction_cor} clearly displays a tree merging procedure which preserves the spectra of the input trees, i.e.\ it constructs an output tree $T_0$ which satisfies Eq.\ (\ref{preserving_formula}), given a finite sequence of input trees $T_1, T_2, \ldots, T_k$.

In the remainder of the paper, we shall provide a full proof of Theorem \ref{tree_construction_th}. We begin with an auxiliary lemma that describes a connection between $\mu(T, \lambda)$ and $\mathcal{G}(r(T), x)$, for an arbitrary rooted tree $T$ and each real number $\lambda \in \mathbb{R}$. The lemma is stated as follows.
\begin{lemma}\label{tree_construction_lemma}
    For any given rooted tree $T$, the assigned rational function $\mathcal{G}(r(T), x)$ from Corollary \ref{adj_cor} corresponding to the root can be represented as a fraction of $\mathbb{Z}[x]$ polynomials
    \begin{align*}
        \mathcal{G}(r(T), x) = \frac{\mathcal{G}_1(r(T), x)}{\mathcal{G}_2(r(T), x)} ,
    \end{align*}
    so that for any real number $\lambda \in \mathbb{R}$, the degree of $\lambda$ as a root of $\mathcal{G}_1(r(T), x)$ is equal to $\mu(T, \lambda)$.
\end{lemma}
\begin{proof}
    We will denote the children of $r(T)$ by $z_1, z_2, \ldots, z_d$, where $d \in \mathbb{N}_0$ is the degree of $r(T)$. We will also signify the rooted subtrees corresponding to these children via $T_1, T_2, \ldots, T_d$, respectively. By applying Corollary \ref{adj_cor}, it is easy to transform Eq.\ (\ref{adj_cor_res}) in order to get
    \begin{align*}
        P(T, x) &= \prod_{v \in V(T)} \mathcal{G}(v, x) \\
        &= \mathcal{G}(r(T), x) \prod_{j = 1}^{d} \left( \prod_{v \in V(T_j)} \mathcal{G}(v, x) \right)\\
        &= \mathcal{G}(r(T), x) \prod_{j = 1}^{d} P(T_j, x) .
    \end{align*}
    From here, we further obtain
    \begin{align*}
        \mathcal{G}(r(T), x) &= \dfrac{P(T, x)}{\displaystyle\prod_{j = 1}^{d} P(T_j, x)} .
    \end{align*}
    If we write
    \begin{align*}
        \mathcal{G}_1(r(T), x) &= P(T, x) ,\\
        \mathcal{G}_2(r(T), x) &= \prod_{j = 1}^{d} P(T_j, x) ,
    \end{align*}
    then $\mathcal{G}_1(r(T), x), \mathcal{G}_2(r(T), x) \in \mathbb{Z}[x]$ and the lemma statement follows directly by taking into consideration that the degree of $\lambda$ as a root of $P(T, x)$ coincides with the value $\mu(T, \lambda)$, for each real number $\lambda \in \mathbb{R}$.
\end{proof}

We are now able to complete the proof of Theorem \ref{tree_construction_th}.

\bigskip
\noindent
{\em Proof of Theorem \ref{tree_construction_th}}.\quad In an analogous manner as done so in the proof of Lemma~\ref{tree_construction_lemma}, we can apply Corollary \ref{adj_cor} and transform Eq.\ (\ref{adj_cor_res}) accordingly in order to obtain
\begin{align}\label{construction_p}
    P(T_0, x) &= \mathcal{G}(r(T), x) \prod_{j = 1}^{k} P(T_j, x)^{\alpha_j} .
\end{align}
From the very construction of the rooted tree $T_0$, it is also straightforward to see that
\begin{align}\label{root_g}
    \mathcal{G}(r(T), x) &= x - \sum_{j = 1}^{k} \dfrac{\alpha_j}{\mathcal{G}(r(T_j), x)} .
\end{align}
Here, we shall apply Lemma \ref{tree_construction_lemma} on each rooted tree from the sequence $T_1, T_2, \ldots, T_k$ in order to get a system of fractional representations
\begin{alignat}{2}\label{frac_impl}
    \mathcal{G}(r(T_j), x) &= \dfrac{\mathcal{G}_1(r(T_j), x)}{\mathcal{G}_2(r(T_j), x)} && \qquad (\forall j = \overline{1, k}) ,\\
    \nonumber \mathcal{G}_1(r(T_j), x), \mathcal{G}_2(r(T_j), x) &\in \mathbb{Z}[x] && \qquad (\forall j = \overline{1, k}),
\end{alignat}
for which the degree of $\lambda$ as a root of $\mathcal{G}_1(r(T_j), x)$ is equal to $\mu(T_j, \lambda)$, for each real number $\lambda \in \mathbb{R}$ and all the $1 \le j \le k$.

By plugging in Eq.\ (\ref{frac_impl}) into Eq.\ (\ref{root_g}), we conclude that
\begin{align}
    \nonumber \mathcal{G}(r(T), x) &= x - \sum_{j = 1}^{k} \dfrac{\alpha_j \ \mathcal{G}_2(r(T_j), x)}{\mathcal{G}_1(r(T_j), x)}\\
    \nonumber &= \dfrac{x \displaystyle\prod_{j=1}^{k} \mathcal{G}_1(r(T_j), x) - \displaystyle\sum_{j=1}^{k} \left( \alpha_j \ \mathcal{G}_2(r(T_j), x) \displaystyle\prod_{i = 1, i \neq j}^{k} \mathcal{G}_1(r(T_i), x) \right)} {\displaystyle\prod_{j=1}^{k} \mathcal{G}_1(r(T_j), x) }\\
    \label{root_g_2} &= \dfrac{\mathcal{G}_1(r(T), x)}{\displaystyle\prod_{j=1}^{k} \mathcal{G}_1(r(T_j), x)} ,
\end{align}
where
\begin{align*}
    \mathcal{G}_1(r(T), x) &= x \displaystyle\prod_{j=1}^{k} \mathcal{G}_1(r(T_j), x) - \displaystyle\sum_{j=1}^{k} \left( \alpha_j \ \mathcal{G}_2(r(T_j), x) \displaystyle\prod_{i = 1, i \neq j}^{k} \mathcal{G}_1(r(T_i), x) \right) .
\end{align*}
We can now use Eq.\ (\ref{root_g_2}) together with Eq.\ (\ref{construction_p}) in order to get
\begin{align}\label{construction_final}
    P(T_0, x) &= \dfrac{\mathcal{G}_1(r(T), x) \displaystyle\prod_{j = 1}^{k} P(T_j, x)^{\alpha_j}}{\displaystyle\prod_{j=1}^{k} \mathcal{G}_1(r(T_j), x)} .
\end{align}

Let $\lambda \in \mathbb{R}$ be an arbitrarily chosen real number. Its degree as a root of each $P(T_j, x)$ is clearly equal to $\mu(T_j, \lambda)$, for each $1 \le j \le k$. This means that the degree of $\lambda$ as a root of the polynomial $\displaystyle\prod_{j = 1}^{k} P(T_j, x)^{\alpha_j}$ must be equal to $\displaystyle\sum_{j = 1}^{k} \alpha_j \ \mu(T_j, \lambda)$. Therefore, the degree of $\lambda$ as a root of the polynomial \begin{equation*}
    \mathcal{G}_1(r(T), x) \displaystyle\prod_{j = 1}^{k} P(T_j, x)^{\alpha_j}
\end{equation*}
must be greater than or equal to $\displaystyle\sum_{j = 1}^{k} \alpha_j \ \mu(T_j, \lambda)$. On the other hand, we know that the degree of $\lambda$ as a root of $\mathcal{G}_1(r(T_j), x)$ must be equal to $\mu(T_j, \lambda)$, for each $1 \le j \le k$. Consequently, the degree of $\lambda$ as a root of the polynomial
\begin{equation*}
    \displaystyle\prod_{j=1}^{k} \mathcal{G}_1(r(T_j), x)
\end{equation*}
must be equal to $\displaystyle\sum_{j = 1}^{k} \mu(T_j, \lambda)$. By taking into consideration Eq.\ (\ref{construction_final}), we see that the degree of $\lambda$ as a root of the polynomial $P(T_0, x)$ must be at least
\begin{align*}
    \sum_{j = 1}^{k} \alpha_j \ \mu(T_j, \lambda) - \sum_{j = 1}^{k} \mu(T_j, \lambda) &= \sum_{j = 1}^{k} (\alpha_j - 1) \mu(T_j, \lambda) .
\end{align*}
This observation immediately gives us
\begin{align*}
    \mu(T_0, \lambda) \ge \sum_{j=1}^{k} (\alpha_j - 1) \mu(T_j, \lambda)
\end{align*}
for every real number $\lambda \in \mathbb{R}$. \qed

\bigskip
\begin{remark}
    Theorem \ref{tree_construction_th} can be proved in an alternative manner by taking into consideration the eigenvectors of the given rooted trees $T_1, T_2, \ldots, T_k$ corresponding to the eigenvalue $\lambda \in \mathbb{R}$. Let $u_1$ be such an eigenvector of $T_1$ and let $T_1^{(1)}, T_1^{(2)}, \ldots T_1^{(\alpha_1)}$ be the $\alpha_j$ rooted subtrees of $T_0$ which correspond to the root children and are isomorphic to $T_1$. For each $2 \le j \le \alpha_1$, we can construct an eigenvector $u_1^{(j)}$ of $T_0$ in the following way:
    \begin{itemize}
        \item each element of $u_1^{(j)}$ corresponding to a vertex from $T_1^{(1)}$ should be equal to the corresponding element of $u_1$;
        \item each element of $u_1^{(j)}$ corresponding to a vertex from $T_1^{(j)}$ should be equal to the additive inverse of the corresponding element of $u_1$;
        \item all the other elements of $u_1^{(j)}$ should be equal to zero.
    \end{itemize}
    
    This means that each eigenvector $u_1$ of $T_1$ corresponding to the eigenvalue $\lambda$ spawns $\alpha_1 - 1$ eigenvectors of $T_0$ corresponding to the same eigenvalue $\lambda$. If we apply the aforementioned construction on $\mu(T_1, \lambda)$ linearly independent eigenvectors of $T_1$, we reach $(\alpha_1 - 1) \mu (T_1, \lambda)$ eigenvectors of $T_0$, which can be shown to be linearly independent via simple computation.
    
    Finally, by implementing the given construction on an arbitrarily chosen set of $\mu(T_j, \lambda)$ linearly independent eigenvectors of the rooted tree $T_j$, for each $1 \le j \le k$, we get $(\alpha_j - 1) \mu (T_j, \lambda)$ linearly independent eigenvectors, for all the $1 \le j \le k$. In total, we obtain $\displaystyle\sum_{j=1}^{k} (\alpha_j - 1) \mu(T_j, \lambda)$ eigenvectors of $T_0$ corresponding to the chosen eigenvalue $\lambda \in \mathbb{R}$. It is not difficult to check that these eigenvectors form a linearly independent set as well, which promptly concludes the alternative proof of Theorem \ref{tree_construction_th}.
\end{remark}

\section*{Acknowledgements}

We are grateful to Dragan Stevanovi\'c for providing the idea behind the alternative proof of Theorem \ref{tree_construction_th}.

\section*{Conflict of interest}

The author declares that he has no conflict of interest.

\end{document}